\documentclass[12pt,a4paper,fleqn]{article}
\input{packages}
\definecolor{dgreen}{rgb}{0,0.5,0}
\definecolor{dblue}{rgb}{0,0,0.5}
\definecolor{dred}{rgb}{0.6,0.0,0.1}
\definecolor{dgold}{rgb}{0.5,0.3,0.0}
\definecolor{dvio}{rgb}{0.6,0.3,0.5}
\definecolor{gray}{rgb}{0.5,0.5,0.5}
\definecolor{dbraun}{rgb}{.5,0.2,0}

\newcommand{\dgrau}{\color{gray}}

\newcommand{\colre}{dred}

\newcommand{\colrem}{dgold}
\newcommand{\colil}{dgreen}
\newcommand{\colde}{dblue}

\newcommand{\lp}[1][]{\ell^{#1}}
\newcommand{\Lp}[1][]{\sL^{#1}}
\newcommand{\norm}[1][]{\lVert\cdot\rVert_{#1}}   
\newcommand{\Vnorm}[2][]{\lVert#2\rVert_{#1}}   
   
\newcommand{\VnormLp}[2][{\Lp[2]}]{\Vnorm[{#1}]{#2}}

\newcommand{\Vskalar}[2][]{\langle #2\rangle_{#1}}   

\newcommand{\VskalarLp}[2][{\Lp[2]}]{\Vskalar[{#1}]{#2}}

\newcommand{\gLp}[1][\ydf]{\Lp[2]_{\mathsmaller{\Rz}}(#1)}
\newcommand{\RLp}{\Lp[2]_{\mathsmaller{\Rz}}}

\def\argmin{\mathop{\mathrm{arg \; min}}\limits}%
\newcommand{\Nsuite}[2][j]{(#2)_{#1 \in\Nz}}
\newcommand{\Zsuite}[2][j]{(#2)_{#1 \in\IZ}}
\newcommand{\set}[1]{{\left\lbrace #1\right\rbrace }}	

\newcommand{\Zset}[2][j\in\Zz]{\set{#2}_{#1}}
\newcommand{\Zsum}[2][j]{\sum_{#1 \in\Zz}#2}
\newcommand{\Nsum}[2][j]{\sum_{#1 \in\Nz}#2}

\newcommand{\nset}[2][]{{\llbracket #1#2\rrbracket }}	
\newcommand{\floor}[1]{\lfloor #1\rfloor}

\newcommand{\lra}{\longrightarrow} 

\newcommand{\eps}{\varepsilon}
\renewcommand{\subset}{\subseteq}
\newcommand{\IN}{\mathbb{N}}

\newcommand{\IZ}{\mathbb{Z}}
\newcommand{\IR}{\mathbb{R}}

\newcommand{\pRz}[1][\Rz]{#1_{+}} 
\newcommand{\IP}{\mathbb{P}}
\newcommand{\IE}{\mathbb{E}}
\newcommand{\iid}{\overset{\text{iid}}{\sim}}

\newcommand{\lcb}{\left\lbrace} 
\newcommand{\rcb}{\right\rbrace} 
\newcommand{\lv}{\left\vert} 
\newcommand{\rv}{\right\vert} 
\newcommand{\lb}{\left(} 
\newcommand{\rb}{\right)} 
\newcommand*{\mc}[1]{\mathcal{#1}}

\newcommand{\dif}{\text{d}}

\newcommand{\Ind}[1]{\mathds{1}\mbox{\scriptsize$#1$}}
\declaretheoremstyle[
    spaceabove=10pt, 
    spacebelow=6pt, 
    headfont=\color{\colre}\normalfont\bfseries,
    notefont=\mdseries\bfseries, 
    notebraces={(}{)}, 
    bodyfont=\normalfont\itshape,
    postheadspace=.3em,
    headpunct={.}]{restyle}

\declaretheoremstyle[
    spaceabove=8pt, 
    spacebelow=8pt, 
    headfont=\color{\colrem}\normalfont\bfseries,
    notefont=\mdseries\bfseries, 
    notebraces={(}{)}, 
    bodyfont=\normalfont\itshape,
    postheadspace=.3em,
    qed=\smaller$\color{\colrem}\square$, 
    headpunct={.}]{remstyle}

\declaretheoremstyle[
    spaceabove=8pt, 
    spacebelow=8pt, 
    headfont=\color{\colde}\normalfont\bfseries,
    notefont=\mdseries\bfseries, 
    notebraces={(}{)}, 
    bodyfont=\normalfont\itshape,
    postheadspace=.3em,
    qed=\smaller$\color{\colde}\square$, 
    headpunct={.}]{destyle}

\declaretheoremstyle[
    spaceabove=8pt, 
    spacebelow=8pt, 
    headfont=\color{\colil}\normalfont\bfseries,
    notefont=\mdseries\bfseries, 
    notebraces={(}{)}, 
    bodyfont=\normalfont\itshape,
    postheadspace=.3em,
    qed=\smaller$\color{\colil}\square$, 
    headpunct={.}]{ilstyle}

\declaretheorem[name=Theorem, style=restyle, numberwithin=section]{theorem}

\declaretheorem[name=Lemma, style=restyle, numberlike=theorem]{lemma}
\declaretheorem[name=Proposition, style=restyle, numberlike=theorem]{proposition}
\declaretheorem[name=Corollary, style=restyle, numberlike=theorem]{corollary}
\declaretheorem[name=Remark, style=remstyle, numberlike=theorem]{remark}

\declaretheorem[name=Illustration, style=ilstyle, numberlike=theorem]{illustration}

\newcommand{\Cz}{{\mathbb C}}

\newcommand{\Ez}{{\mathbb E}}

\newcommand{\Nz}{{\mathbb N}}

\newcommand{\Pz}{{\mathbb P}}

\newcommand{\Rz}{{\mathbb R}}

\newcommand{\Zz}{{\mathbb Z}}

\newcommand{\sL}{\mathscr{L}}

\newcommand{\cD}{{\mathcal D}}
\newcommand{\cE}{{\mathcal E}}

\newcommand{\cK}{{\mathcal K}}

\newcommand{\oA}{\overline{A}}

\newcommand{\uA}{\underline{A}}

\newcommand{\tS}{\widetilde{S}}
\newcommand{\tT}{\widetilde{T}}

\newcommand{\hS}{\widehat{S}}
\newcommand{\hT}{\widehat{T}}

\newcommand{\xOb}[1][\nb]{X_{#1}}
\newcommand{\yOb}[1][\nb]{Y_{#1}}
\newcommand{\epsOb}[1][\nb]{\varepsilon_{#1}}

\newcommand{\xden}[1][]{f_{#1}}
\newcommand{\xoden}[1][]{f^{\circ}_{{#1}}}
\newcommand{\yoden}[1][]{g^{\circ}_{{#1}}}
\newcommand{\yden}[1][]{g_{#1}}
\newcommand{\epsden}[1][]{\varphi_{#1}}

\newcommand{\ccon}{\text{\textcircled{$\star$}}}

\newcommand{\say}{n}

\newcommand{\qF}{\mathrm{q}^2} 
 
\newcommand{\hqF}{\hat{\mathrm{q}}^2} 
 
\newcommand{\qFr}{{\mathrm{q}}} 
 
\newcommand{\tqF}{\tilde{\mathrm{q}}^2} 
\newcommand{\test}{\Delta}

\newcommand{\sera}{\rho}

\newcommand{\msera}[1][]{\sera_n}

\newcommand{\dimcollect}{\mc K}

\newcommand{\densities}{\mc D}

\newcommand{\expan}{j}
\newcommand{\wdclass}[1][\expan]{a_{#1}}
\newcommand{\rdclass}{\mathrm{R}}
\newcommand{\altcollect}{\mc A}
\newcommand{\adaptivefactor}{\delta_n}

\newcommand{\expb}[1][\expan]{\mathrm{e}_{\mathsmaller{#1}}}

\newcommand{\Pf}[1][f]{\IP_{#1}} 
 
\newcommand{\trisk}[1][]{\mc R(#1) }

\newcommand{\chisq}{\chi^2}
\newcommand{\signv}[1][]{\tau_{\hspace*{-.2ex}\msml[.7]{#1}}}
\newcommand{\signvb}[1][]{\eta_{\msml[.7]{#1}}}

\newcommand{\sPara}{s}
\newcommand{\pPara}{p}

\newcommand{\indexset}{\mc S}
\newcommand{\indexone}{s}
\newcommand{\indextwo}{t}
\newcommand{\dimindexone}{k^\indexone}
\newcommand{\dimindextwo}{k^\indextwo}

\newcommand{\ssconst}{\mathfrak a} 

\usepackage{ACC-abrev-package}
\author{{\sc Sandra Schluttenhofer}\;\thanks{Institut f\"ur Angewandte
  Mathematik, M$\Lambda$THEM$\Lambda$TIKON, Im Neuenheimer Feld 205,
  D-69120 Heidelberg, Germany, e-mail:
 \url{{schluttenhofer|johannes}@math.uni-heidelberg.de}} \and {\sc Jan Johannes}$\;^*$}

\date{Ruprecht-Karls-Universität Heidelberg} 
\title{Adaptive minimax testing for circular convolution} 

\begin{document}
\maketitle
\begin{abstract}
Given observations from a circular random variable contaminated by an additive measurement error, we consider the problem of minimax optimal goodness-of-fit testing in a non-asymptotic framework. We propose direct and indirect testing procedures using a projection approach. The structure of the optimal tests depends on regularity and ill-posedness parameters of the model, which are unknown in practice. Therefore, adaptive testing strategies that perform optimally over a wide range of regularity and ill-posedness classes simultaneously are investigated. Considering a multiple testing procedure, we obtain adaptive i.e. assumption-free procedures and analyse their performance. Compared with the non-adaptive tests, their radii of testing face a deterioration by a log-factor. We show that for testing of uniformity this loss is unavoidable by providing a lower bound. The results are illustrated considering Sobolev spaces and ordinary or super smooth error densities. 
\end{abstract}
{\footnotesize
	\begin{tabbing} 
		\> circular data, adaptive testing, deconvolution, goodness-of-fit, Bonferroni aggregation
	\\[.2ex] 
\end{tabbing}}%
%

%
%
%
%
\section{Introduction}
\paragraph{The statistical model.}
We consider a circular convolution model where a random variable that takes values on the circle is observed contaminated by an additive error. Identifying the circle with the unit interval $[0,1)$, the observable random variable can be expressed as
\begin{align}
	\label{model}
	\yOb[] = \xOb[] + \epsOb[] - \floor{\xOb[] + \epsOb[]},
\end{align}
where $\xOb[]$ and $\epsOb[]$ are independent random variables supported on the interval $[0,1)$ and $\floor{\cdot}$ denotes the floor-function. The aim of this paper is to investigate adaptive testing procedures for the circular density $\xden$ of $X$ given a sample of independent and identically distributed (iid) copies of $Y$. If $\epsden$ denotes the density of the error $\eps$, then the observable random variable $\yOb[]$ admits a density $\yden = \xden \ccon \epsden$, where $\ccon$ denotes circular convolution defined by
\begin{align*}
	\xden \ccon \epsden (y) = \int_{[0,1)} \xden\lb y - s - \floor{y-s} \rb \epsden(s) \dif s, \qquad y \in [0,1).
\end{align*}
Hence, making inference on $\xden$ based on observations from $\yden$ is a deconvolution problem. Circular, wrapped (around the circumference of the unit circle), spherical or directional data appear in various applications. We briefly mention two popular fields. Circular models are used for data with a temporal or periodic structure, where the circle is identified e.g. with a clock face (cp. \cite{GillHangartner2010}). Moreover, identifying the circle with a compass rose, directional data can also be represented by a circular model.  For many more examples of circular data we refer the reader to \cite{Mardia1972}, \cite{Fisher1995} and \cite{MardiaJupp2009}. \cite{KerkyacharianNgocPicard2011a} and \cite{LacourNgoc2014}, for instance, investigated a circular model with multiplicative error. Nonparametric estimation in the additive error model \eqref{model} has amongst others been considered in \cite{Efromovich1997}, \cite{ComteTaupin2003} and \cite{JohannesSchwarz2013a}.

\paragraph{Testing task.}
 For a prescribed density $\xoden$ we test the null hypothesis $\lcb
 \xden = \xoden \rcb$ against the alternative $\lcb \xden \neq \xoden
 \rcb$ based only on observations of $Y$. We separate the null
 hypothesis and the alternative to make them distinguishable. Consider
 the Hilbert space $\Lp[2] := \Lp[2]([0,1))$ of square-integrable
 complex-valued functions on $[0,1)$ equipped with its usual  norm
 $\norm[{\Lp[2]}]$. We assume throughout this paper that both $\xden$
 and $\epsden$ (and, hence, $\yden$) belong to the subset of real
 probability densities $\densities$ in $\Lp[2]$. Let $\lcb \yOb[j]
 \rcb_{j=1}^n$ be $\say$ independent and identically distributed
 copies of $Y$, i.e. the observations are given by
 \begin{align}
   \label{observations}
   \lcb \yOb[j] \rcb_{j=1}^n \iid \yden = \xden \ccon \epsden. 
 \end{align}
 Denote by $\FuVg{\xdf}$ and $\FuEx{\xdf}$ the probability
 distribution and the expectation associated with the data
 \eqref{observations}, respectively. For a separation radius
 $\sera \in \pRz$, let us define the \textit{energy} set
 $\Lp[2]_{\sera} := \lcb \xden \in \Lp[2] : \VnormLp{{\xden}} \geq
 \sera \rcb$. For a nonparametric class of functions $\mc E$,
 capturing the \textit{regularity} of the alternative, the testing
 task can be written as
 \begin{align}
   \label{testing:problem}
   H_0: \xden = \xoden \qquad \text{against} \qquad H_1^{\sera}:
   \xden - \xoden  \in \Lp[2]_{\sera} \cap \mc E, \xdf\in\cD.
 \end{align}
 In the literature there exist several definitions of rates and radii
 of testing in an asymptotic and nonasymptotic sense. The classical
 definition of an asymptotic rate of testing for nonparametric
 alternatives is essentially introduced in the series of papers
 \cite{Ingster1993}, \cite{Ingster1993a} and \cite{Ingster1993b}. For
 a fixed sample size, two alternative definitions of a nonasymptotic
 radius of testing are typically considered. For prescribed error
 probabilities $\alpha,\beta\in(0,1)$, \cite{Baraud2002},
 \cite{LaurentLoubesMarteau2012} and \cite{MarteauSapatinas2017},
 amongst others, define a nonasymptotic radius of testing as the
 smallest separation radius $\rho$ such that there is an $\alpha$-test
 with maximal type II error probability over the $\rho$-separated
 alternative smaller than $\beta$.  The definition we use in this
 paper -- which is based on the sum of both error probabilities -- is
 adapted e.g. from \cite{CollierCommingesTsybakov2017}.  We measure
 the accuracy of a test $\test$, i.e. a measurable function
 $\test: \IR^n \to \lcb 0,1 \rcb$, by its maximal risk defined as the
 sum of the type I error probability and the maximal type II error
 probability over the $\sera$-separated alternative
 \begin{align*}
   \trisk[\test \mid \mc E, \sera] := \FuVg{\xdfO}(\test = 1) + \sup_{\substack{{\xdf\in\cD\,:\,}\xdf- \xdfO \in \Lp[2]_{\sera} \cap \mc E}} \FuVg{\xdf}(\test = 0). 
 \end{align*}
 We are particularly interested in the smallest possible value of
 $\sera$ for which the null and the $\sera$-separated alternative
 are still distinguishable. A value $\sera^2(\mc E)$ is called an
 upper bound of the \textit{radius of testing} for a family of tests
 $\{\tFi[\alpha], \alpha\in(0,1)\}$ over the alternative $\mc E$, if for all
 $\alpha \in (0,1)$ there exists a constant   
 $\oA_{\alpha} \in \pRz$ such that
 \begin{enumerate}
 \item[] for all $A \geq \oA_{\alpha}$ we have $\trisk[{\tFi[\alpha]} \mid \mc E, A \sera(\mc E)] \leq \alpha$, \hfill (upper bound)
 \end{enumerate}
 The difficulty of the testing problem can be characterised by the
 minimax risk
 \begin{align*}
   \trisk[ \mc E, \sera] := 	\inf_{\test} \trisk[\test \mid \mc E, \sera] 
 \end{align*}
 where the infimum is taken over all possible tests. The value
 {$ \sera^2(\mc E)$} is called \textit{minimax radius of
   testing}, if in
addition for all $\alpha \in (0,1)$ there exists a constant $\uA_{\alpha} \in \pRz$ such that
 \begin{enumerate}
 \item[] for all $A \leq \uA_{\alpha}$ we have $\trisk[ \mc E, A \sera(\mc E) ] \geq 1 - \alpha$ \hfill (lower bound)
 \end{enumerate}
 and the family $\{\tFi[\alpha],\alpha\in(0,1)\}$ is then called minimax optimal.

 \paragraph{Direct and indirect testing procedures.} Considering the
 density $\xoden = \mathds{1}_{[0,1)}$ of a uniform distribution only,
 minimax radii of testing in the circular model \eqref{model} are for
 example derived in \cite{SchluttenhoferJohannes2020a} for
 nonparametric alternatives $\mc E$, covering Sobolev spaces and
 ordinary or super smooth error densities. The authors consider a test
 that is based on a projection estimator of the quantity
 $\Vnorm[{\Lp[2]}]{\xdf -\xdfO}^2$, depending on a dimension
 parameter.  The test, roughly speaking, compares the estimator to a
 multiple of its standard deviation. Choosing the dimension parameter
 optimally, they show that this test is minimax optimal, i.e. it
 achieves the minimax radius of testing given by a typical
 bias-variance trade-off. However, estimating
 $\Vnorm[{\Lp[2]}]{\xdf - \xdfO}^2$ based on observations
 \eqref{observations} in deconvolution models is an inverse problem,
 since it requires an inversion of the convolution
 transformation. This inversion introduces additional instability in
 deconvolution problems, caused by its ill-posedness. To circumvent
 this problem, in an inverse Gaussian sequence space model
 \cite{LaurentLoubesMarteau2011} argue for a \textit{direct} testing
 procedure, which is based on the estimation of the energy in the
 image space of the operator. Let us explain this idea in our
 setting. Instead of the \textit{indirect testing task}
 \eqref{testing:problem} we examine the \textit{direct testing task}
 of the null hypothesis $\lcb \yoden = \yden \rcb$ with
 $\yoden := \xoden \ccon \epsden$ against the alternative
 $\lcb \yoden \neq \yden \rcb$, where we have direct access to
 observations from $\yden$. Using similar arguments as
   \cite{SchluttenhoferJohannes2020a}, it is possible to derive the
   testing radii of a direct test, based on a projection estimator of
   the \textit{direct} quantity $\Vnorm[{\Lp[2]}]{\ydf - \ydfO}^2$.
 In both the direct and indirect testing procedure the radii of
 testing, depending on the dimension parameter, are essentially
 determined  by a bias-variance trade-off. As usual the optimal
 choice of the dimension parameter depends on both the smoothness of
 the alternative and the ill-posedness of the model, which are unknown
 in practice. This motivates the study of adaptive testing procedures,
 which we investigate in this paper.
	
 \paragraph{Adaptive testing.} In the literature adaptive,
 i.e. assumption-free, testing strategies have been studied in both an
 asymptotic and a nonasymptotic framework. In an asymptotic framework,
 e.g. \cite{Spokoiny1996} considers adaptive testing strategies in a
 sequence space model with Besov-type alternatives, showing that
 asymptotic adaptation comes with an unavoidable cost of a
 $\log$-factor. In a nonasymptotic setting,
 \cite{LaurentHuetBaraud2003} consider adaptive testing in a Gaussian
 regression model, \cite{FromontLaurent2006} deal with a density
 model. \cite{Butucea2007} and \cite{ButuceaMatiasPouet2009} determine
 adaptive rates of testing in a convolution model on the real line
 using kernel estimators of the $\Lp[2]$-distance to the null. The
 proposed tests have as a common feature that they are based on
 estimators of the distance to the null, which only depend on the
 (unknown) smoothness through a tuning parameter (e.g. a bandwidth, a
 threshold or a dimension parameter). By aggregating the estimators
 over different tuning parameters into one test statistic - i.e. using
 a multiple testing approach - the authors obtain tests, which perform
 optimally over a wide range of alternatives. Since they no longer
 depend on the unknown regularity of the alternative, they are
 \textit{assumption-free}. To formalise this idea, let us introduce a
 collection $\altcollect$ of regularity parameters that characterise a
 family of alternatives $\{ \wCc,\wC \in \altcollect\}$ with
 corresponding radii
 $\{\tRiN{n} := \sera(\wCc),\wC \in \altcollect\}$, where we now
 explicitly emphasise the dependence on the regularity parameter
 $\wC\in \altcollect$ and the number of observations $n$ in the
 notation. In general, adaptation without a loss is impossible
 (cp. \cite{Spokoiny1996}). To characterise the cost to pay for
 adaptation, we introduce the \textit{effective sample size}
 $\adaptivefactor n$. The factor $\adaptivefactor \in [0,1]$ shrinks
 the sample size $n$ and, hence, evaluating the radius at
 $\adaptivefactor n$ deteriorates the radius of testing.  In fact, the
 value $\adaptivefactor^{-1}$ is called \textit{adaptive factor} for
 the family of tests $\{\tFi[\alpha], \alpha\in(0,1)\}$ over the family of alternatives
 $\{\wCc, \wC \in \altcollect\}$, if for all
 $\alpha \in (0,1)$ there exists a constant $\uA_\alpha$ such that
	\begin{itemize}
		\item[] for all $A \geq \oA_{\alpha}$ we have
                  $\sup_{\wC \in \altcollect} \trisk[{\tFi[\alpha]}
                  \mid \wCc, A \tRiN{\adaptivefactor n}] \leq \alpha$, \hfill 
	\end{itemize}
	where $\tRiN{ n}$ denotes a radius of testing
        of the family $\tFi[\alpha]$, $\alpha\in(0,1)$. We shall emphasise that the testing risk now has to be bounded uniformly for all alternatives $\wCc$. We call $\adaptivefactor^{-1}$ \textit{minimal adaptive factor} if for all $\alpha \in (0,1)$ there exists a constant $\oA_\alpha \in \IR$ such that
	\begin{itemize}
		\item[] for all $A \leq \oA_{\alpha}$ we have $\inf_{\test} \sup_{\wC \in \altcollect} \trisk[\test \mid  \wCc, A \tRiN{\adaptivefactor n}] \geq  1 - \alpha$. \hfill 
	\end{itemize}

\paragraph{Aggregation procedure.} Let us come back to the circular deconvolution problem and the indirect and direct tests discussed above. In this paper we aggregate both testing procedures over a family $\dimcollect \subseteq \IN$ of dimension parameters using a classical Bonferroni method, where for a given level $\alpha \in (0,1)$ each of the tests in the family has level $\frac{\alpha}{\lv \mc K \rv}$. The aggregated testing procedure  rejects the null hypothesis as soon as one test in the collection rejects. It is straight-forward to see that a Bonferroni aggregation of the indirect test proposed in \cite{SchluttenhoferJohannes2020a} leads to an adaptive factor of order $\lv \mc K \rv$. The choice of the family $\mc K$ reflects the collection of alternatives, over which the aggregated test performs optimally. If the alternatives characterise ordinary smoothness of the circular density, the size of $\mc K$ is typically chosen to be of order $\log n$ (c.p. \cite{FromontLaurent2006}, \cite{Spokoiny1996}). Then the aggregated test will feature a deterioration by an adaptive factor of order $\log n$. However, we show in this paper that generally the  minimal adaptive factor is smaller. In order to do so, we first derive sharper bounds for the quantiles of the direct and indirect test statistics using exponential bounds for U-statistics and a Bernstein inequality. This allows to define a new version of an indirect and a direct test, for which we derive radii of testing.  Aggregating these tests via the Bonferroni method, we obtain an adaptive factor for adaptation with respect to smoothness of order $\sqrt{\log \log n}$. Interestingly, for testing for uniformity, i.e. $\xoden = \mathds{1}_{[0,1)}$, the aggregated direct test does not depend on the noise density $\epsden$ and is, thus, also adaptive with respect to the ill-posedness of the model. Moreover, in this situation, we derive a lower bound for the adaptive factor, which provides conditions under which it is minimal.
\paragraph{Outline of the paper.}  
The upper bounds for the radius of testing via an indirect and a
direct testing procedure are derived in \cref{sec:2,sec:4},
respectively. \cref{sec:a:u:b,sec:a:u:b:d} are devoted to adaptive
indirect and direct testing strategies. We provide lower bounds in
\cref{sec:l:b}. Technical derivations are deferred to the
\cref{proofs1}.

\section{Upper bound via an indirect testing procedure}\label{sec:2}
\paragraph{Notation.} The inner product on $\Lp[2]$ that induces the
norm $\norm[{\Lp[2]}]$ is given by
$\Vskalar[{\Lp[2]}]{\HeB,\He}=\int_{[0,1)}\HeB(x)\overline{\He(x)}dx$,
where $\overline{\He(x)}$ denotes the complex conjugate of $\He(x)$.
Consider the family of exponential functions $\Zset{\expb}$ with
$\expb(x) := \exp(-\i 2\pi \expan x)$ for $x \in [0,1)$ and
$\expan\in\Zz$, which is an orthonormal basis of $\Lp[2]$.
Consequently, any $\He\in\Lp[2]$ admits an expansion as a discrete Fourier
series  $\He =
\sum_{\expan \in \IZ} \fHe[j]\expb$ with
$\fHe[j]:=\Vskalar[{\Lp[2]}]{\He,\expb}$, for $j\in\Zz$. By Parseval's
identity  its sequence of Fourier coefficients $\fHe:=\Zsuite{\fHe[j]}$
is square summable. In fact, setting
$\Vnorm[{\lp[p]}]{\fHe}:=(\Zsum{|\fHe[j]|^p})^{1/p}$  for $p\geq1$ we
have $\Vnorm[{\lp[2]}]{\fHe}=\Vnorm[{\Lp[2]}]{\He}$. Moreover, for a density $\ydf$
let us further denote by $\gLp$ the set of
all real-valued (Borel-measurable) functions $h$ satisfying
$\int_{[0,1)}h^2(x)\ydf(x)dx<\infty$, in particular,  $\RLp:=\gLp[\mathds{1}_{[0,1)}]$.
\paragraph{Definition of the test statistic.}
We expand the densities $\xdf,\xdfO\in\densities\subset\Lp[2]$
appearing in the testing task \eqref{testing:problem}
in the
exponential basis. Therefore we have   $ \Vnorm[{\Lp[2]}]{\xdf-\xdfO}^2=
\Nsum[|j|]|\fxdf[j] -
\fxdfO[j]|^2=2\Nsum[j]|\fxdf[j] -
\fxdfO[j]|^2$ due to
Parseval's  identity, since $\fxdf[0]=1=\fxdfO[0]$,
$\fxdf[j]=\ofxdf[-j]$ and  $\fxdfO[j]=\ofxdfO[-j]$ for all $j\in\Zz$.  Additionally, exploiting the circular convolution
theorem, the density $\ydf = \xdf \ccon \edf$ of the observations
\eqref{observations}   admits Fourier coefficients $\fydf[\expan] =
\fxdf[\expan] \cdot \fedf[\expan]$ for all $\expan \in \IZ$.
Keeping $\ydfO = \xdfO \ccon \edf$ and  $Y\sim \ydf=\xdf \ccon \edf$ in mind and  assuming from here on $|\fedf[\expan]|>0$ for all
$\expan\in\Zz$,  we have
\begin{equation*}
	\qF(\xdf-\xdfO): =\Nsum[|j|]|\fxdf[j] -
\fxdfO[j]|^2 = \Nsum[|j|]{\frac{|\fydf[j]- \fydfO[j]|^2
  }{|\fedf[j]|^2}}\quad\text{ with
  }\fydf[j]=\FuEx{\xdf}(\expb(-Y))\text{ for all }j\in\Zz.
\end{equation*}
For $k\in\Nz$ and $\nset{k}:=[1,k]\cap\Nz$  let us define an unbiased estimator $\hqF_k$ of the truncated version
\begin{equation}	\label{qfk}
\qF_k(\xdf-\xdfO):=\sum_{|j|\in\nset{k}} |\fxdf[j]-\fxdfO[j]|^2 =
\sum_{|j|\in\nset{k}}\frac{|\fydf[j]|^2}{|\fedf[j]|^{2}}-2
\sum_{|j|\in\nset{k}}\frac{\fydfO[j]\ofydf[j]}{|\fedf[j]|^{2}}+
  \sum_{|j|\in\nset{k}}|\fxdfO[j]|^2
\end{equation}
using  that
$\sum_{|j|\in\nset{k}}|\fedf[j]|^{-2}\fydfO[j]\ofydf[j]$ is
real-valued.
Replacing  the
unknown Fourier coefficients by empirical counterparts based on observations $\{\yOb[j]\}_{j=1}^n$ we consider the
test statistic
\begin{multline}	\label{hqfk}
\tSi:= \hT_k-2\hS_k + \qF_k(\xdfO)\quad \text{ with }\\
 \hfill
 \hT_k:=\frac{1}{n(n-1)}\sum_{|j|\in\nset{k}}\sum_{l,m\in\nset{n}\atop
 l\ne m} \frac{\expb(-Y_l)\expb(Y_m)}{|\fedf[j]|^2}\quad\text{ and }\quad\hfill\\
  \hS_k:=\frac{1}{n}\sum_{|j|\in\nset{k}} \sum_{l\in\nset{n}}\frac{\fydfO[j]\,\expb(Y_l)}{|\fedf[j]|^{2}}.
\end{multline}
Note that  $\qF_k(\xdfO)$ is known, $\hT_k$ is  a U-statistic
and $\hS_k$ is a linear statistic.
We shall emphasise, if $\{\yOb[j]\}_{j=1}^n \iid \yden $ as in
\eqref{observations} then  $\tSi$ is an unbiased estimator of $\qF_k(\xdf-\xdfO)$ for each
 $k\in\Nz$. Below we construct  a test that, roughly speaking, compares
 the estimator $\tSi$ to a multiple of its standard deviation.
\paragraph{Decomposition of the test statistic.} The key element to
analyse quantiles of the test statistic $\tSi$  in \eqref{hqfk} is
the following decomposition
\begin{multline}\label{decomposition}
	\tSi  = \uSi  + 2\lSi + \qF_k(\xden - \xoden)\quad\text{
          with }\\
\hfill   \uSi := \frac{1}{n(n-1)}
\sum_{|j|\in\nset{k}}\sum_{l,m\in\nset{n}\atop l\ne m} \frac{ (\expb(-Y_l) - \fydf[j])(\expb(Y_m) - \ofydf[j])}{|\fedf[j]|^2}\quad\text{and}\hfill\\
   	 \lSi := \frac{1}{n} \sum_{|j|\in\nset{k}} 
   \sum_{l\in\nset{n}}\frac{(\fydf[j]-\fydfO[j])(\expb(Y_l)-\ofydf[j])}{|\fedf[j]|^2},     
\end{multline}
where $\qF_k(\xdf - \xdfO)$ is a  separation term, $\uSi$ is  a
canonical  U-statistic
and $\lSi$ is a centred linear statistic.
\paragraph{Definition of the threshold.} 
The next proposition provides bounds for the quantiles of the test statistic  $\tSi$. Let $L_x := (1 -
\log x)^{1/2} \in (1,\infty)$ for $x \in (0,1)$.  Define for $k \in \IN$  the quantities%
\begin{equation}\label{nu}
	\nu_k:= \big(\sum_{|j|\in\nset{k}}
  |\fedf[j]|^{-4}\big)^{1/4} \quad\text{and}\quad m_k := \max_{j\in\nset{k}} |\fedf[j]|^{-1}.
\end{equation}
For $\alpha \in (0,1)$ with $c_1:= 799 \Vnorm[{\lp[2]}]{\fydfO}+ 1372,
c_2 := 52  \Vnorm[{\lp[1]}]{\fydfO}$ consider the threshold
\begin{equation}\label{tau}
\tSiC :=  c_1  \lb 1 \vee
  L_\alpha^2\frac{\nu_k}{n^{1/2}} \vee L_\alpha^{3}\frac{\nu_k^2}{n}
  \rb {L_\alpha} \frac{\nu^2_k}{n} + c_2 L^2_\alpha \frac{m^2_k}{n}.
\end{equation}
Note that  $\Vnorm[{\lp[1]}]{\fydfO}\leq
\Vnorm[{\Lp[2]}]{\xdfO}\Vnorm[{\Lp[2]}]{\edf}< \infty$ due to the
Cauchy-Schwarz inequality and Parseval's identity. 
\begin{proposition}\label{quantiles}
For 
  $\xdfO,\xdf,\edf\in\Lp[2]$ and $n\in\Nz$, $n\geq2$, consider $\lcb \yOb[j] \rcb_{j=1}^n \iid \ydf=\xdf\oast\edf$
  with joint distribution $\FuVg{\xdf}$ and
  $\ydfO=\xdfO\oast\edf$. Let $\alpha, \beta \in (0,1)$ and for
  $k\in\Nz$ consider $\tSi$ and $\tSiC$ as in
  \eqref{hqfk} and \eqref{tau}, respectively.
  \begin{resListeN}
  \item\label{alpha:quantile} If   $\gLp[\ydfO]=\RLp$, then 
    $	\FuVg{\xdfO}\lb \tSi \geq \tSiC  \rb \leq \alpha.$
  \item\label{beta:1} If
    $c_3:=8\Vnorm[{\lp[1]}]{\fydf}+826\Vnorm[{\lp[2]}]{\fedf}^2 +1372$ and
    \begin{equation}
      \label{separation}
      \qF_k(\xden - \xoden) \geq  2\big( \tSiC  +
      c_3 L_{\beta/2}^4 \big( 1  \vee
      \frac{\nu_k^2}{n}\big)\frac{\nu_k^2}{n}\big)                    ,
    \end{equation}	
    then	$ \FuVg{\xdf} \lb \tSi < \tSiC \rb \leq \beta.$
  \end{resListeN}
\end{proposition}
\begin{proof}[Proof of \cref{quantiles}]
  Firstly, consider \ref{alpha:quantile}. If $\xdf=\xdfO$ and, hence $\yden = \yoden$, the
  decomposition \eqref{decomposition} simplifies to $\tSi =
  \uSi$, where $\uSi$ is a canonical U-statistic. Applying
  \cref{simplified:con} given in the appendix, a concentration inequality for
  canonical U-statistics of order 2, with $x =
  L_{\alpha}^2\geq1$ and quantities $\uSA$-$\uSD$ satisfying
  \eqref{U:const} we obtain
  \begin{align}
    \label{alpha:1}
   \FuVg{\xdfO} \lb \tSi  \geq  8 \uSC n^{-1} L_\alpha
    + 13 \uSD n^{-1}  L_\alpha^2 + 261 \uSB n^{-3/2} L_\alpha^{3} +
    343 \uSA n^{-2} L_\alpha^4 \rb \leq \alpha. 
  \end{align}
  Consider  the quantities $\uSA$-$\uSC$ in  \eqref{ABCD} and $\uSD$ in
  \eqref{D:null}, which satisfy \eqref{U:const} under the additional condition
  $\gLp[\ydfO]=\RLp$  due
  to \cref{prop:u}.  We have $  8 \uSC n^{-1} L_\alpha
  + 13 \uSD n^{-1}  L_\alpha^2 + 261 \uSB n^{-3/2} L_\alpha^{3}
  + 343 \uSA n^{-2} L_\alpha^4\leq\tSiC $. This  together with \eqref{alpha:1} shows \ref{alpha:quantile}. 
  Secondly, consider \ref{beta:1}. Keeping the decomposition
  \eqref{decomposition} in mind we control the deviations of the
  U-statistic $\uSi$ and the linear statistic $\lSi$ by  applying   \cref{simplified:con}
  and \cref{linear},   respectively. In fact, the
  quantities $\uSA$-$\uSD$ given in   \eqref{ABCD} of \cref{prop:u} fulfil
  (recall $L_{\beta/2}\geq1$) 
  \begin{multline*} 8
    \uSC n^{-1}L_{\beta/2} + 13 \uSD n^{-1} L_{\beta/2}^2 + 261 \uSB n^{-3/2}
    L_{\beta/2}^3 + 343 \uSA
    n^{-2} L_{\beta/2}^4\\
    \leq L_{\beta/2}^4 (825 \Vnorm[{\lp[2]}]{\fydf} +1372)(1\vee
    (\nu_k^2n^{-1}) ) \nu_k^2n^{-1}=:\tau_1.
  \end{multline*}
  Consequently, the event $\Omega_1: = \{\uSi \leq -\tau_1\}$ satisfies
  $\FuVg{\xdf}(\Omega_1)\leq\beta/2$ due to \cref{simplified:con}
  (with the usual symmetry argument). Define further the event
  $\Omega_2 := \lcb 2\lSi \leq - \tau_2 - \tfrac{1}{2} \qF_k(\xden - \xoden )
  \rcb$ with
  $\tau_2:=L_{\beta/2}^2 (8\Vnorm[{\lp[1]}]{\fydf}+\Vnorm[{\lp[2]}]{\fedf}^2) (1\vee
  (m_k^2n^{-1}) ) m_k^2n^{-1}$,
  then
  $\Pf[\xden](\Omega_2) \leq \frac{\beta}{2e} \leq \frac{\beta}{2}$
  due to \cref{linear} with $x = L_{\beta/2}\geq1$, which
  is an application of a Bernstein inequality.
  Since $ \tau_1 + \tau_2\leq L_{\beta/2}^{4} c_3 (1\vee
  (\nu_k^2n^{-1}) ) \nu_k^2n^{-1}$ with
  $c_3=8\Vnorm[{\lp[1]}]{\fydf}+826\Vnorm[{\lp[2]}]{\fedf}^2 +1372$  due to $m_k^2 \leq \nu_k^2$,
  $1\leq L_{\beta/2}$ and
  $\Vnorm[{\lp[2]}]{\fydf}\leq\Vnorm[{\lp[2]}]{\fedf}^2$  
  the assumption \eqref{separation} yields
  $\tfrac{1}{2} \qF_k(\xden - \xoden) \geq \tSiC + \tau_1 +
  \tau_2$. Thus, the decomposition \eqref{decomposition} implies
  \begin{multline*}
    \FuVg{\xdf}\lb \tSi < \tSiC \rb  = \FuVg{\xdf} \lb \lcb  \tSi < \tSiC  \rcb \cap \Omega_1 \rb + \FuVg{\xdf} \lb \lcb  \tSi < \tSiC  \rcb \cap \Omega_1^c \rb  \\
    \leq \FuVg{\xdf}\lb \Omega_1 \rb + \Pf[\xden]\lb 2\lSi+ \qF_k(\xden - \xoden) < \tSiC + \tau_1\rb  \leq \beta/2 + \FuVg{\xdf}(\Omega_2)\leq\beta, 
  \end{multline*}
  which shows \ref{beta:1} and completes the proof.
\end{proof}
\begin{remark}
  The technical assumption $\gLp[\ydfO]=\RLp$ in
  \cref{quantiles} allows us to express elements of $\gLp[\ydfO]$
  in their Fourier expansion. It is immediately satisfied for
  $\xoden =\expb[0]= \mathds{1}_{[0,1)}$ and if $\yoden$ is bounded away from
  $0$ and infinity.
\end{remark}
\paragraph{Definition of the test.} For $k\in\Nz$ and $\alpha\in(0,1)$
using the test statistic
$\hqF_k$  and the threshold $\tSiC$ given in \eqref{hqfk}
and \eqref{tau}, respectively,  we consider the test
\begin{align}
	\label{test}
	\tFi := \Ind{\{\tSi \geq \tSiC\}}. 
\end{align}
 From \ref{alpha:quantile} in \cref{quantiles} it immediately
 follows that $\tFi$ is a level-$\alpha$-test for all $k
 \in \IN$. To analyse its power over the alternative, we introduce a
 regularity constraint, i.e. a  nonparametric class of functions $\mc
 E$, which is
 formulated in terms of  Fourier coefficients. Let $R > 0$ and let $\wC=
 \Nsuite{\wC[j]}$ be a strictly positive, monotonically
 non-increasing sequence that is bounded by $1$. We assume that the
 difference $\xdf-\xdfO$  belongs to the $\Lp[2]$-ellipsoid
\begin{align}
	\rwC= \lcb \He \in\Lp[2]: 2 \sum_{\expan \in \IN  }  \wdclass^{-2} \lv \fHe[\expan] \rv^2 \leq \rdclass^2 \rcb. \label{ellipsoid} 
\end{align}
Note that $\xdf-\xdfO \in \rwC$ imposes conditions on all coefficients $\xden[\expan]$, $j \in \IZ$, since $\lv f_j \rv^2 =\lv {f_{-j}} \rv^2$, $j \in \IN$, for all real-valued functions and, additionally, $f_0 = 1$ for all densities. 
The definition \eqref{ellipsoid} is general enough to cover classes of
ordinary and super smooth functions.  \cref{quantiles} \ref{beta:1}
allows to characterise elements in $\rwC$ for which $\tFi$
is powerful. Exploiting these results, in the next proposition we
derive an upper bound for the radius of testing of $\tFi$
in terms of $\nu_k$ as in \eqref{nu} and the regularity parameter
$\wC$, that is, we define
 \begin{align*}
   \tSRiK{k} :=\tSRiKN{k}{n}:= \wC[k]^2 \vee \frac{\nu_k^2}{n} .
 \end{align*}
\begin{proposition}
  \label{re:ub}
  Under the assumptions of \cref{quantiles}  
	for  $\alpha \in (0,1)$ define
        \begin{equation}
\oA_\alpha^2 := R^2 +  2 (8R\Vnorm[{\lp[2]}]{\fedf}+826\Vnorm[{\lp[2]}]{\fedf}^2+
    859\Vnorm[{\lp[1]}]{\fydfO}+2744) L_{\alpha/4}^4\label{re:ub:A}.
  \end{equation}
For all  $A \geq \oA_\alpha$ and for all  $n, k \in \IN$ with 
$\nu_k^2 \leq n$ we have 
	$\mc R(\tFi[k,\alpha/2] \mid \rwC, A \tRiK{k}) \leq \alpha$.
\end{proposition}
\begin{proof}[Proof of \cref{re:ub}]
  We apply \cref{quantiles}  to show that both the type I and the
  maximal type II error probability are bounded
  by $\alpha/2$, then the result follows immediately from the definition of the risk
  $\trisk[{\tFi[k,\alpha/2]} \mid \rwC, A \tRiK{k}] =
  \FuVg{\xdfO}({\tFi[k,\alpha/2]} = 1) + \sup_{\substack{ \xden -
      \xoden \in \rwC\cap\Lp[2]_{A \tRiK{k}} }}
  \FuVg{\xdf}({\tFi[k,\alpha/2]} = 0)\leq \alpha/2+\alpha/2=\alpha$.
  Since the assumption of \cref{quantiles} \ref{alpha:quantile}  is  fulfilled the test $\tFi[k,\alpha/2]$ is a level-$\alpha/2$-test. Hence,   in order to apply
  \cref{quantiles} \ref{beta:1} (with
  $\beta = \alpha/2$)  for  each density $\xdf\in\Lp[2]$ with
  $\Vnorm[{\Lp[2]}]{\xdf-\xdfO}^2 \geq \oA_\alpha^2
  \tSRiK{k}$ and $\xdf - \xdfO \in \rwC$ 
 it remains  to
  verify condition \eqref{separation} which  states
     $ \qF_k(\xdf - \xdfO) \geq  2\big( \tSiCa{{\alpha}/{2}}  +
      c_3 L_{\alpha/4}^4 ( 1  \vee
      \nu_k^2n^{-1}) \nu_k^2n^{-1}\big)              $
    with $\tSiCa{{\alpha}/{2}} $  as in \eqref{tau}. Indeed, in this situation we have
  $ \sum_{\lv j \rv > k} \lv \xden[j] - \xoden[j] \rv^2 \leq
  \wdclass[k]^2 \rdclass^2$, which implies 
  \begin{multline}\label{re:ub:bew:e1}
    \qF_k(\xden -\xoden)  = \Vnorm[{\Lp[2]}]{\xdf-\xdfO}^2-
    \sum_{\lv j \rv > k} \lv \xden[j] - \xoden[j] \rv^2
    \geq \oA_\alpha^2 \tSRiK{k} - \wdclass[k]^2 \rdclass^2 \\\hfill
    \geq2(8R\Vnorm[{\lp[2]}]{\fedf}+826\Vnorm[{\lp[2]}]{\fedf}^2+
    859\Vnorm[{\lp[1]}]{\fydfO}+2744) L_{\alpha/4}^4 {\nu_k^2}n^{-1}\hfill\\
    \geq 2(8\Vnorm[{\lp[1]}]{\fydf}+826\Vnorm[{\lp[2]}]{\fedf}^2+
    851\Vnorm[{\lp[1]}]{\fydfO}+2744) L_{\alpha/4}^4 {\nu_k^2}n^{-1},
  \end{multline}
  using the triangular inequality
  $\Vnorm[{\lp[1]}]{\fydf}\leq\Vnorm[{\lp[1]}]{\fydf-\fydfO}+\Vnorm[{\lp[1]}]{\fydfO}$
  and the Cauchy-Schwarz inequality  $\Vnorm[{\lp[1]}]{\fydf-\fydfO}\leq\Vnorm[{\lp[2]}]{\fxdf-\fxdfO}\Vnorm[{\lp[2]}]{\fedf}\leq
  R\Vnorm[{\lp[2]}]{\fedf}$ by $\Vnorm[{\lp[2]}]{\fxdf-\fxdfO}\leq R$.
  The  condition \eqref{separation} follows 
  from \eqref{re:ub:bew:e1} by
  exploiting further  $1\leq L_{\alpha/2}\leq L_{\alpha/4}$, $\Vnorm[{\lp[2]}]{\fydfO}\leq\Vnorm[{\lp[1]}]{\fydfO}$ and 
  $m_k^2\leq \nu_k^2\leq n$, which completes the proof. 
\end{proof}
Let us introduce a dimension that  realises an optimal bias-variance
trade-off and the corresponding radius
\begin{multline}\label{tDi}
  \tDi := \argmin\limits_{k \in \IN} \tSRiK{k} := \min \lcb k \in
  \IN: \tSRiK{k} \leq \tSRiK{l} \text{ for all } l \in \IN
  \rcb\quad\text{ and}\\
  \tSRi:=\tSRiN{n}:=  \min_{k \in \IN}\tSRiKN{k}{n}= \min_{k \in \IN} \{ a_k^2 \vee \nu_k^2n^{-1}  \}.
\end{multline}
\begin{corollary}
	\label{re:m:ub}
  Under the assumptions of \cref{quantiles}  
 let  $\alpha \in (0,1)$ and $\oA_\alpha$ as in \eqref{re:ub:A}, 
then 	$\mc R(\tFi[\tDi,\alpha/2] \mid \rwC, A \tRi) \leq \alpha$ for all $A \geq \oA_\alpha$ and  $n \geq \sqrt{2} \lv \epsden[1] \rv^{-2}$.
\end{corollary}
\begin{proof}[Proof of \cref{re:m:ub} ]The result follows immediately
  from \cref{re:ub}, since $\nu_{\tDi}^2\leq n$ for all $n \geq
  \sqrt{2} \lv \epsden[1] \rv^{-2}$. Indeed, $n \geq \sqrt{2}
  \lv \epsden[1] \rv^{-2}$ implies $1\geq \tSRiK{1}\geq \tSRi\geq
  \nu_{\tDi}^2 n^{-1}$.
\end{proof}
We shall emphasise that in the case $\xdfO =
\expb[0]=\mathds{1}_{[0,1)}$ the radius of testing $\tRi$ is known to
be  minimax (\cite{SchluttenhoferJohannes2020a}) and, hence, the test $\tFi[\tDi,\alpha/2] $ is minimax optimal.

\begin{illustration}\label{ill}
  Throughout the paper we illustrate the order of the radii of testing
  under typical regularity and ill-posedness assumptions.  For two
  real-valued sequences $(x_j)_{j \in \IN}$ and $(y_j)_{j \in \IN}$ we
  write $x_j \lesssim y_j$ if there exists a constant $c > 0$ such
  that $x_j \leq c y_j$ for all $j \in \IN$. We write $x_j \sim y_j$,
  if both $x_j \lesssim y_j$ and $y_j \lesssim x_j$. Concerning the
  class $\rwC$ we distinguish two behaviours of the sequence $\wC$,
  namely the \textbf{ordinary smooth} case
  $\wC[j] \sim { j^{-\sPara}}$ for $\sPara > 1/2$, corresponding to a
  Sobolev ellipsoid, and the \textbf{super smooth} case
  $\wC[j] \sim {\exp(-j^{\sPara})}$ for $\sPara > 0$, corresponding to
  a class of analytic functions.  We also distinguish two cases for
  the regularity of the error density $\epsden[]$. For $\pPara>1/2$ we
  consider a \textbf{mildly ill-posed} model
  $\lv \epsden[j] \rv \sim \lv j \rv ^{-\pPara}$ and for $p > 0$ a
  \textbf{severely ill-posed} model
  $ \lv \epsden[j] \rv \sim {\exp(-\lv j \rv^{\pPara})}$.  Many
  examples of circular densities can be found in Chapter 3 of
  \cite{MardiaJupp2009}.  The table below presents the order of the
  dimension $\tDi$ and the upper bound $\tSRi$ for the radius of testing (see the appendix of
  \cite{SchluttenhoferJohannes2020a} for the calculations).\\ [3ex]
  \centerline{\begin{tabular}{ll|l|l}\toprule \multicolumn{4}{c}{Order
        of the optimal dimension $\tDi$ and the radius
        $\tSRi$}\\ \midrule
                ${\wC[j]}$ & ${\lv \fedf[j] \rv}$ & $\tDi$ & $\tSRi$   \\
                (smoothness) & (ill-posedness) & & \\
                \midrule
                ${ j^{-\sPara}}$ & ${\lv j \rv^{-\pPara}}$ &  $n^{\tfrac{2}{4\pPara + 4\sPara + 1}}$ & $\say^{-\tfrac{4\sPara}{4\sPara + 4 \pPara + 1}}$  \\
                ${ j^{-\sPara}}$ & 	$ {e^{-\lv j \rv^{\pPara}}}$ & $(\log n)^{\tfrac{1}{\pPara}}$ & $(\log \say)^{-\tfrac{2\sPara}{\pPara}}$  \\
                ${e^{-j^{\sPara}}}$ & $ {\lv j \rv^{-\pPara}}$& $ (\log n)^{\tfrac{1}{\sPara}}$ & $\say^{-1} (\log n)^{\tfrac{2 \pPara + 1/2}{\sPara}}$   \\
                \bottomrule
 				\end{tabular}}
  		\end{illustration}

\section{Adaptive  indirect testing procedure}
\label{sec:a:u:b}
  For an arbitrary regularity parameter $\wC$ the test
  $\tFi[\tDi,\alpha/2]$ in \cref{re:m:ub} achieves the minimax radius
  of testing $\tSRi$. However, it relies via the dimension parameter
  $\tDi$ on the regularity class $\rwC$ and, thus, the testing
  procedure is not \textit{adaptive},
  i.e. \textit{assumption-free}. Ideally, a test should perform
  optimally for a wide range of regularity parameters. In this section
  we therefore propose an adaptive testing procedure by aggregating
  the test in \eqref{test} over various dimension
  parameters. 
  \paragraph{Adaptation procedure via Bonferroni aggregation.}  Let
  $\mc K \subseteq \IN$ be a finite collection of dimension
  parameters. For $k \in \cK$ and the level
  $\alpha/|\cK|\in (0,1)$ recall the
  collection of tests
  $(\tFi[k,\alpha/|\cK|])_{k \in\cK} = (\Ind{\{\tSi > \tSiCa{\alpha/|\cK|}\}})_{k \in \mc K}$ defined in
  \eqref{test}. 
  We
  consider the $\max$-test
\begin{align}\label{de:tFia}
	\tFia := \Ind{\{ \tSiM > 0\}} \qquad \text{ with } \qquad \tSiM: = \max_{k \in \mc K} \lb \tSi - \tSiCa{\tfrac{\alpha}{|\cK|}}\rb,
\end{align}
	i.e. the test rejects the null hypothesis as soon as one of
        the tests in the collection does.
        Under the null hypothesis, we bound the type I error probability of the $\max$-test by the sum of the error probabilities of the individual tests,
	\begin{align}
		\label{adapt:null}
		\Pf[\xoden](\tFi[\cK,\alpha]=1) & = \Pf[\xoden](\tSiM > 0) \leq \sum_{k \in \mc K} \Pf[\xoden](\tFi[k,\alpha/|\cK|]=1) \leq \sum_{k \in \mc K} \frac{\alpha}{|\cK|} = \alpha.
	\end{align}
	Hence, $\tFi[\cK,\alpha]$ is a level-$\alpha$-test, since
 $\tFi[k,\alpha/|\cK|]$ is a level-$\alpha/|\cK|$-test       for
  each $k \in \cK$ due to \cref{quantiles} \ref{alpha:quantile}.   Under the alternative, we can bound the type II error probability by the error probability of any of the individual tests,
	\begin{align}
		\label{adapt:alternative}
		\Pf[\xden](\tFi[\cK,\alpha]= 0) = \Pf[\xden](\tSiM \leq  0) \leq \min_{k \in \mc K} \Pf[\xden](\tFi[k,\alpha/|\cK|]=0).
	\end{align}
	Therefore, $\tFi[\cK,\alpha]$ has the maximal power achievable
        by a test in the collection. The bounds \eqref{adapt:null} and
        \eqref{adapt:alternative} have opposing effects on the choice
        of the collection $\mc K$. On the one hand, it should be as
        small as possible to keep the type I error probability
        small. On the other hand, it must be large enough to contain
        an optimal dimension parameter $\tDi$ for a wide range of
        regularity parameters $\wC$, that we want to adapt to. In this
        paper we consider a classical Bonferroni choice of an error
        level $\alpha/|\cK|$. For other aggregation choices, e.g. a
        Monte-Carlo quantile and a Monte-Carlo threshold method we
        refer to \cite{LaurentHuetBaraud2003} and
        \cite{FromontLaurent2006}. Although the Bonferroni choice is a
        more conservative method, we  show its optimality, which is
        than shared with the other methods.

\paragraph{Testing radius of the indirect $\max$-test.}
Denote by $\mc A \subseteq \IR^\IN_{> 0}$ a set of strictly positive,
monotonically non-increasing sequences bounded by $1$. The set $\mc A$
characterises the collection of regularity classes  $\lcb \rwC: \wC
\in \mc A \rcb$, for which  the power of the testing procedure is
analysed simultaneously.
The $\max$-test $\tFia$ in \eqref{de:tFia}  only aggregates over
a finite set $\mc K \subset \IN$. For each 
$\wC \in \altcollect$ we define a minimal achievable
radius of testing $\tSRiaN{n}$  over the set $\mc K$
as
\begin{equation*}
  \tSRiaN{x} := \min_{k \in \mc K}\tSRiKN{k}{x} \quad\text{ with
        }\quad\tSRiKN{k}{x}:=  a_k^2 \vee \frac{\nu_k^2}{x} \quad\text{ for all }x\in\pRz. 
\end{equation*}
Since $\tSRiN{n}=\tSRiKN{\Nz\,}{n}$ in \eqref{tDi} is defined as the  minimum taken over $\IN$
instead of $\mc K$, for  $n\in\Nz$ we always have $\tSRiaN{n} \geq
\tSRiN{n}$. Moreover, replacing $\nu_k$ by $m_k$ as in \eqref{nu}, let
us define a remainder radius $\tSReaN{n}$, typically negligible
compared to $\tSRiaN{n}$, as follows
\begin{equation}\label{tSReaN}
        \tSReaN{x}:= \min_{k \in \mc K} \tSReaKN{x} \quad\text{ with
        }\quad\tSReaKN{x}:=  a_k^2 \vee \frac{m_k^2}{x}  \quad\text{ for all  }x\in\pRz. 
\end{equation}

\begin{proposition}[Uniform radius of testing over $\altcollect$]
  \label{adapt:ub}
  Under the assumptions of \cref{quantiles}  
  let  $\alpha \in (0,1)$ and consider $\oA_\alpha$ as in \eqref{re:ub:A}.
    Then, for all $A \geq \oA_\alpha$ and $n\in\Nz$
  \begin{equation*}
    \sup_{\wC \in \altcollect} \trisk[{\tFia[\cK,\alpha/2]}
    \mid\rwC,A(\tReaN{\delta^2n}\vee \tRiaN{\delta n})(1\vee\delta^{-3/2}\tRiaN{\delta n})]\leq\alpha
  \end{equation*} 
  with $\delta=  (1+\log \lv \mc K \rv)^{-1/2}$.
\end{proposition}
\begin{proof}[Proof of \cref{adapt:ub}] For each $\wC\in\altcollect$
  we apply \cref{quantiles} to show that both the type I and the
  maximal type II error probability are bounded by $\alpha/2$, then
  the result follows immediately from the definition of the risk.
  Under the null hypothesis, the claim follows from \eqref{adapt:null}
  together with 
  \cref{quantiles} \ref{alpha:quantile}  and
  $\sum_{k \in \mc K} \frac{\alpha}{2|\cK|} = \alpha/2$.  Under the
  alternative, let $\xdf\in\Lp[2]$ with $\xdf-\xdfO\in
  \rwC$ satisfy
  \begin{equation}
    \label{adapt:ub:bew:e1}
    \Vnorm[{\Lp[2]}]{\xdf-\xdfO}^2 \geq \oA_\alpha^2(\tSReaN{\delta^2n}\vee \tSRiaN{\delta n})(1\vee\delta^{-3}\tSRiaN{\delta n}).
  \end{equation}
  It sufficient to use the elementary  bound \eqref{adapt:alternative}
  together with the observation that
  \begin{resListeN}
  \item\label{adapt:ub:bew:c1} $ \FuVg{\xdf}(
    \tSi < \tSiCa{\tfrac{\alpha}{2|\cK|}})\leq \alpha/2$ holds for all $k\in\cK$  whenever
    $\xdf\in\Lp[2]$ satisfies $\xdf-\xdfO\in
  \rwC$ and
  \begin{equation}
    \label{adapt:ub:bew:c1:e}
    \Vnorm[{\Lp[2]}]{\xdf-\xdfO}^2 \geq \oA_\alpha^2 \lb \wC[k]^2\vee(1 \vee
  \frac{\nu_k}{\delta^2n^{1/2}} \vee \frac{\nu_k^2}{\delta^3n}
  ) \frac{\nu^2_k}{\delta n} \vee \frac{m^2_k}{\delta^2n}\rb;
  \end{equation}
\item \label{adapt:ub:bew:c2} under \eqref{adapt:ub:bew:e1}  there exists $k\in\cK$ such that   \eqref{adapt:ub:bew:c1:e}
  is fulfilled. 
\end{resListeN}    
As a consequence,  we
have $\FuVg{\xdf}(\tFia[\cK,\alpha/2] = 0)\leq \alpha/2$ for all
$\xdf\in\Lp[2]$ satisfying  $\xdf-\xdfO\in
  \rwC$ and \eqref{adapt:ub:bew:e1}, and thus  the maximal type II
  error probability is also bounded  by
$\alpha/2$. It remains to show \ref{adapt:ub:bew:c1} and
\ref{adapt:ub:bew:c2}. The claim \ref{adapt:ub:bew:c1} follows from
\cref{quantiles} \ref{beta:1} (with  $\beta = \alpha/2$)  since
\eqref{adapt:ub:bew:c1:e} implies the
  condition \eqref{separation},   which  states
     $ \qF_k(\xdf - \xdfO) \geq  2\big( \tSiCa{\tfrac{\alpha}{2|\cK|}}  +
      c_3 L_{\alpha/4}^4 ( 1  \vee
      \nu_k^2n^{-1}) \nu_k^2n^{-1}\big)              $
    with $\tSiCa{\tfrac{\alpha}{2|\cK|}} $  as in \eqref{tau}. Indeed,  exploiting  $L_{\alpha/2}^2=\log(2e/\alpha)\geq1$ and hence
$L_{\alpha/(2|\cK|)}^2= \log |\cK|+  L_{\alpha/2}^2\leq
L_{\alpha/2}^2(1+\log |\cK|)= L_{\alpha/2}^2\delta^{-2}$, we have 
\begin{multline*}
\tSiCa{\tfrac{\alpha}{2|\cK|}} \leq c_1 L_{\alpha/2}^4  \lb 1 \vee
  \frac{\nu_k}{\delta^2n^{1/2}} \vee \frac{\nu_k^2}{\delta^3n}
  \rb \frac{\nu^2_k}{\delta n} + c_2 L^2_{\alpha/2}
  \frac{m^2_k}{\delta^2n}\\\leq (c_1+c_2)L_{\alpha/2}^4  \lb (1 \vee
  \frac{\nu_k}{\delta^2n^{1/2}} \vee \frac{\nu_k^2}{\delta^3n}
  ) \frac{\nu^2_k}{\delta n} \vee \frac{m^2_k}{\delta^2n}\rb.
\end{multline*}
Using $L_{\alpha/4} \geq L_{\alpha/2}$,
$1 \geq \delta$ and $\oA_\alpha^2-\rdclass^2\geq
2(c_1+c_2+c_3)L_{\alpha/4}^4$ by elementary calculations
the condition \eqref{separation} holds whenever
\begin{equation}\label{cond}
  \qF_k(\xdf - \xdfO) \geq  \big(\oA_\alpha^2-\rdclass^2\big)
  \lb (1 \vee
  \frac{\nu_k}{\delta^2n^{1/2}} \vee \frac{\nu_k^2}{\delta^3n}
  ) \frac{\nu^2_k}{\delta n} \vee \frac{m^2_k}{\delta^2n}\rb.
\end{equation}	
Due to  $\xdf-\xdfO\in
  \rwC$ and hence $ \sum_{\lv j \rv > k}  |\fxdf[j] - \fxdfO[j]|^2 \leq
  \wdclass[k]^2 \rdclass^2$,  the condition \eqref{adapt:ub:bew:c1:e} implies
 \begin{equation*}
    \qF_k(\xden -\xoden)  = \Vnorm[{\Lp[2]}]{\xdf-\xdfO}^2-
    \sum_{\lv j \rv > k} |\fxdf[j] - \fxdfO[j]|^2\geq
    \big(\oA_\alpha^2-\rdclass^2\big)
  \lb (1 \vee
  \frac{\nu_k}{\delta^2n^{1/2}} \vee \frac{\nu_k^2}{\delta^3n}
  ) \frac{\nu^2_k}{\delta n} \vee \frac{m^2_k}{\delta^2n}\rb.
\end{equation*}
Consequently, if \eqref{adapt:ub:bew:c1:e} is satisfied, then also \eqref{cond} and thus \eqref{separation}, which shows the claim
\ref{adapt:ub:bew:c1}. Lastly, consider \ref{adapt:ub:bew:c2}. By Lemma A.1 in
\cite{SchluttenhoferJohannes2020} we have  $\tSReaN{\delta^2n}\vee \tSRiaN{\delta
  n}
=\wC[k]^2\vee
\tfrac{m_k^2}{\delta^{2}n}\vee \tfrac{\nu_k^2}{\delta n}$ and $\tSRiaN{\delta
  n}\geq \tfrac{\nu_k^2}{\delta n}$ for  at least one $k \in
\cK$. 
Hence, there is $k\in\cK$ with $\big(\tSReaN{\delta^2n}\vee\tSRiaN{\delta n}\big)(1\vee\tfrac{\tSRiaN{\delta
      n}}{\delta^{3}})\geq
  \wC[k]^2\vee\tfrac{m_k^2}{\delta^{2}n}\vee\big( \tfrac{\nu_k^2}{\delta
    n}(1\vee \tfrac{\nu_k^2}{\delta^{4}n})\big)$. Since $1 \vee
 \tfrac{\nu_k^2}{\delta^4n}\geq1 \vee
  \frac{\nu_k}{\delta^2n^{1/2}} \vee \frac{\nu_k^2}{\delta^3n}
  $, this shows
  \ref{adapt:ub:bew:c2} and completes the proof.
\end{proof}

\begin{corollary}[Worst-case adaptive factor]
  \label{cor:worst} Under the assumptions of \cref{quantiles} let
  $\alpha \in (0,1)$ and consider  $\oA_\alpha$ as in \eqref{re:ub:A}.
     Then, for all $A \geq
  \oA_\alpha$ and $n\in\Nz$
  \begin{equation*}
    \sup_{\wC \in \altcollect} \trisk[{\tFia[\cK,\alpha/2]} \mid \rwC, A
    \tRiaN{\delta^2 n}(1\vee\tRiaN{\delta^2 n})]  \leq \alpha
  \end{equation*}
 with $\delta=  (1+\log \lv \mc K \rv)^{-1/2}$.
\end{corollary}
\begin{proof}[Proof of \cref{cor:worst}]The result may be proved in
  much the same way as  \cref{adapt:ub}.   In fact,  as in the proof of \cref{adapt:ub}, it is sufficient to show
  \ref{adapt:ub:bew:c2} replacing \eqref{adapt:ub:bew:e1} by
    \begin{equation}
    \label{cor:worst:bew:e1}
    \Vnorm[{\Lp[2]}]{\xdf-\xdfO}^2 \geq \oA_\alpha^2\tSRiaN{\delta^2
      n}(1\vee\tSRiaN{\delta^2 n}).
  \end{equation}
 For each
 $\wC\in\altcollect$  under  \eqref{cor:worst:bew:e1}  
  the parameter  $k:=\argmin\nolimits_{k' \in \mc
   K}\tSRiKN{k'}{\delta^2 n}$ satisfies
  \begin{multline*}
    \tSRiaN{\delta^2 n}(1\vee\tSRiaN{\delta^2 n})=
    \tSRiaN{\delta^2 n}(1\vee\tRiaN{\delta^2
      n}\vee\tSRiaN{\delta^2 n})\geq
    \wC[k]^2\vee\tfrac{\nu_k^2}{\delta^{2}n}(1\vee\tfrac{\nu_k}{\delta
      n^{1/2}} \vee\tfrac{\nu_k^2}{\delta^{2}n})\\\geq
    \wC[k]^2\vee\tfrac{m_k^2}{\delta^{2}n}\vee\tfrac{\nu_k^2}{\delta
      n}(1\vee\tfrac{\nu_k}{\delta^2 n^{1/2}}
    \vee\tfrac{\nu_k^2}{\delta^{3}n})
  \end{multline*}
  since $\wC[k]^2\vee\tfrac{\nu_k^2}{\delta^{2}n}=\tSRiaN{\delta^2 n}$, $\delta\leq1$ and $m_{k}^2 \leq
  \nu_{k}^2$. This shows \eqref{adapt:ub:bew:c1:e}  and, in consequence, \ref{adapt:ub:bew:c2}. We
  obtain the assertion  proceeding exactly as
  in the proof of \cref{adapt:ub}. 
\end{proof}

By
\cref{cor:worst}, $\tSRiaN{\delta^2 n}$ is an upper bound
for the radius of testing of the  indirect $\max$-test if $\tSRiaN{\delta^2 n}\leq 1$. The latter   is satisfied
for an arbitrary regularity parameter $\wC\in \altcollect$, if $1
\in  \mc K$ and $ n \geq \sqrt{2} \lv \epsden[1] \rv^{-2}(
1+\log|\cK|) $.  The next corollary establishes $\tSRiaN{\delta n}$ as
a sharper upper bound for the radius of testing of the  indirect $\max$-test
under additional conditions, which are satisfied in all the examples considered in \cref{ill:adapt}
below.
The result follows
immediately from \cref{adapt:ub} and we omit its proof.

\begin{corollary}[Best-case adaptive factor]
  Under the assumptions of \cref{quantiles} let $\alpha \in (0,1)$ and
  consider $\oA_\alpha$ as in \eqref{re:ub:A}. If there exists a constant $C\geq1$
  such that $\tReaN{\delta^2n}\leq C \tRiaN{\delta n}$ and
  $\tRiaN{\delta n}\leq C \delta^{3/2}$ for all
  $\wC \in \altcollect$, then for all
  $A \geq C^2 \overline{A}_{\alpha}$ and $n\in\Nz$
  \begin{equation*}
    \sup_{\wC \in \altcollect} \trisk[{\tFia[\cK,\alpha/2]} \mid \rwC, A\tRiaN{\delta n} ]  \leq \alpha
  \end{equation*}
    with $\delta=  (1+\log \lv \mc K \rv)^{-1/2}$.
\end{corollary}

\begin{remark}[Choice of $\mc K$]\label{choice:of:K}Ideally,  the collection $\cK\subset{\Nz}$ is
  chosen such that its elements approximate  the optimal parameter
  $\tDi$ given in \eqref{tDi} sufficiently well  for each $\wC\in\altcollect$. 
  Note that $\tDi\leq n^2/2$ for all $\wC\in\altcollect$ and  for $n$ reasonably large (precisely
  $n \geq \sqrt{2} \lv \epsden[1] \rv^{-2}$, which implies
  $1\geq \tSRiK{1}\geq \tSRi\geq \nu_{\tDi}^2 n^{-1}\geq(\tDi)^{1/2}n^{-1}$).  Hence, a naive choice is
  $\cK = \nset{{n^2}/{2}}$ with
  $|\cK|=\lfloor {n^2}/{2} \rfloor$, which yields an adaptive factor
  of order $(\log n)^{1/2}$.  However, in most cases, a minimisation
  over a geometric grid
  $\cK_{g}= \{1\}\cup\{2^j,j\in\nset{\log_2(n^2/2)}\}$
  approximates the minimisation over $\IN$
   well enough. Since $|\cK_g| = \lfloor \log_2(n^2/2) \rfloor$
   the adaptive factor is then of order $(\log\log n)^{1/2}$.
   For some special cases the 
   smaller collection 
 	$\cK_s =\{1\}\cup\{2^j, j\in\nset{ s^{-1}\log_2\log n}\}$, $s>0$, is still sufficient (see
        \cref{ill:adapt} below), resulting in  an adaptive factor
  of order $(\log\log\log n)^{1/2}$.
\end{remark}

\begin{illustration}\label{ill:adapt} For the typical configurations
  for regularity and ill-posedness introduced in \cref{ill} the tables
  below display the adaptive radii  of the $\max$-test
    $\tFia[\cK,\alpha/2]$ for appropriately chosen grids. In a mildly
  ill-posed model with ordinary smoothness we choose the geometric
  grid $\cK_g$ and, hence, $\delta=(1+\log |\cK_g|)^{-1/2}$. It is
  easily seen that the remainder term $\tSReacKN{\cK_g}{\delta^2 n}$
  is asymptotically negligible compared with
  $\tSRiKN{\cK_g}{\delta n} $.  Moreover,
  $\delta^{-3}\tSRiKN{\cK_g}{\delta n}$ tends to zero.  Hence, the
  upper bound in \cref{adapt:ub} asymptotically reduces to
  $\tSRiKN{\cK_g}{\delta n}$ featuring an adaptive factor of order
  $(\log\log n)^{1/2}$.\\[3ex]
  \centerline{%
  \begin{tabular}{ll|l|l}\toprule
    \multicolumn{4}{l}{Order of $\tSReacKN{\cK_g}{\delta^2
    n}$ and $\tSRiKN{\cK_g}{\delta n}$ with  $\delta=(1+\log
    |\cK_g|)^{-1/2}$}\\
    \multicolumn{4}{r}{and  $\cK_g=
    \{1\}\cup\{2^j,j\in\nset{\log_2({n^2}/{2})}\}$} \\
    \midrule
    ${\wC[j]}$ & ${|\fedf[j]|}$ & $\tSReacKN{\cK_g}{\delta^2 n}$
    &$\tSRiKN{\cK_g}{\delta n}$   \\
    (smoothness) & (ill-posedness) &  \\
    \midrule
    ${ j^{-\sPara}}$ & ${\lv j \rv^{-\pPara}}$
                                & $\lb \frac{\say}{{\log\log n}}
                                  \rb^{-\tfrac{4\sPara}{4\sPara + 4
                                  \pPara}}$
    & $\lb \frac{\say}{(\log\log n)^{1/2}} \rb^{-\tfrac{4\sPara}{4\sPara + 4 \pPara + 1}}$  \\
    \bottomrule
  \end{tabular}}\\[3ex]
In a severely ill-posed model with ordinary smoothness, we have seen
in \cref{ill} that the order of the optimal dimension parameter does
not depend on the smoothness parameter. Hence, the test
$\tFi[\tDi,\alpha/2]$ is automatically adaptive with respect to
ordinary smoothness. In a mildly ill-posed model with super
smoothness, we choose the smaller geometric grid $\cK_{\sPara_\star}$
for adaptation to smoothness $\sPara\geq\sPara_\star$
and, hence,  $\delta=(1+\log
|\cK_{\sPara_\star}|)^{-1/2}$. It is easily seen that the remainder term
$\tSReacKN{\cK_{\sPara_\star}}{\delta^2 n}$   is asymptotically negligible
compared with $\tSRiKN{\cK_{\sPara_\star}}{\delta n} $.  
Moreover,  
 $\delta^{-3}\tSRiKN{\cK_{\sPara_\star}}{\delta n}$ tends to
 zero. Hence, the upper bound in \cref{adapt:ub}  asymptotically reduces to
 $\tSRiKN{\cK_{\sPara_\star}}{\delta n}$ featuring an adaptive factor of order
 $(\log\log\log n)^{1/2}$.\\ [3ex]
 \centerline{%
   \begin{tabular}{ll|l|l}\toprule
     \multicolumn{4}{l}{Order of $\tSReacKN{\cK_{\sPara_\star}}{\delta^2
     n}$ and $\tSRiKN{\cK_{\sPara_\star}}{\delta n}$ with  $\delta=(1+\log
     |\cK_{\sPara_\star}|)^{-1/2}$}\\
     \multicolumn{4}{r}{and  $\cK_{\sPara_\star}=
     \{1\}\cup\{2^j,j\in\nset{ {s^{-1}_{\star}} \log_2 \log n			}\}$} \\
     \midrule
     ${\wC[j]}$ & ${|\fedf[j]|}$ & $\tSReacKN{\cK_{\sPara_\star}}{\delta^2 n}$ & $\tSRiKN{\cK_{\sPara_\star}}{\delta n}$    \\
     (smoothness) & (ill-posedness) &  \\
     \midrule
     ${ e^{-j^\sPara}}$ & ${\lv j \rv^{-\pPara}}$  & $ \frac{{\log\log\log n}}{n} (\log n)^{\frac{2 \pPara}{\sPara}}$  &  $ \frac{(\log\log\log n)^{1/2}}{n}  (\log n)^{\frac{2 \pPara + 1/2}{\sPara}}$  \\
     \bottomrule
   \end{tabular}}
\end{illustration}

\begin{remark}[Adaptation to the radius $\rdclass$ of the regularity class]
In this paper the parameter $\rdclass > 0$ is unknown but assumed to be fixed and
we consider adaptation to a collection of alternatives $\{ \rwC: \wC
\in \altcollect\}$ only.  From \cref{re:m:ub} (and the definition of
$\oA_\alpha$ therein) it follows immediately that adaptation to $\{
\rwC: \rC \in (0, \rC^\star] \}$ is achieved without a
loss. Indeed,  replacing $\rC$ by $\rC^\star$ in the definition of
$\oA_\alpha$ we readily obtain a result similar to \cref{re:m:ub}
with an additional supremum taken over $R \in (0, \rC^\star]$.
However, adaptation to $\{ \rwC: \rC\in(0,\infty) \}$ is not possible
without a loss, for an explanation we refer to Section 6.3 of \cite{Baraud2002} for a similar observation in the Gaussian sequence space model. 
\end{remark}

\section{Upper bound via a direct testing procedure}\label{sec:4}
\paragraph{Definition of the test statistic.} In this section we
consider a test that is based on an estimation of the quantity
\begin{equation*}
	\qF(\ydf-\ydfO) = \Nsum[|j|]|\fydf[j]- \fydfO[j]|^2=\Nsum[|j|]|\fxdf[j] -
\fxdfO[j]|^2|\fedf[j]|^2\quad\text{ with
  }\fydf[j]=\FuEx{\xdf}(\expb(-Y))\text{ for all }j\in\Zz.
\end{equation*}
For $k\in\Nz$  we define an unbiased estimator $\tqF_k$ of the truncated version
\begin{equation}	\label{qfgk}
\qF_k(\ydf-\ydfO)=\sum_{|j|\in\nset{k}} |\fydf[j]- \fydfO[j]|^2 =
\sum_{|j|\in\nset{k}}|\fydf[j]|^2-2
\sum_{|j|\in\nset{k}}\fydfO[j]\ofydf[j]+
\sum_{|j|\in\nset{k}}|\fydfO[j]|^2
\end{equation}
using again  that
$\sum_{|j|\in\nset{k}}\fydfO[j]\ofydf[j]$ is
real-valued. Replacing  the
unknown Fourier coefficients by empirical
counterparts based on observations $\{\yOb[l]\}_{l=1}^n$, we consider the
test statistic
\begin{multline}	\label{tqfk}
\tSd:= \tT_k-2\tS_k + \qF_k(\ydfO)\quad \text{ with }\\
\hfill
\tT_k:=\frac{1}{n(n-1)}\sum_{|j|\in\nset{k}}\sum_{l,m\in\nset{n}\atop l\ne m}\expb(-Y_l)\expb(Y_m)\quad\text{ and }\quad\hfill\\
  \tS_k:=\frac{1}{n}\sum_{|j|\in\nset{k}} \sum_{l\in\nset{n}}\fydfO[j]\,\expb(Y_l).
\end{multline}
Note that  $\qF_k(\ydfO)$ is known, $\tT_k$ is  a U-statistic
and $\tS_k$ is a linear statistic. Moreover, if $\{\yOb[j]\}_{j=1}^n \iid \yden $ as in
\eqref{observations} then  $\tqF_k$ is an unbiased projection estimator of $\qF_k(\ydf-\ydfO)$ for each $k\in\Nz$. 
\paragraph{Decomposition of the test statistic.} We analyse quantiles
of the test statistic $\tqF_k$  in \eqref{tqfk} using the decomposition
\begin{multline}\label{d:decomposition}
	\tSd  = \uSd + 2\lSd + \qF_k(\ydf - \ydfO)\quad\text{
          with }\\
\hfill   \uSd:= \frac{1}{n(n-1)}
\sum_{|j|\in\nset{k}}\sum_{l,m\in\nset{n}\atop l\ne m} 
 (\expb(-Y_l) - \fydf[j])(\expb(Y_m) - \ofydf[j])\quad\text{and}\hfill\\
   	 \lSd := \frac{1}{n} \sum_{|j|\in\nset{k}} 
   \sum_{l\in\nset{n}}(\fydf[j]-\fydfO[j])(\expb(Y_l)-\ofydf[j]),     
\end{multline}
where $\qF_k(\ydf - \ydfO)$ is a  separation term, $\uSd$ is  a
canonical U-statistic
and $\lSd$ is a centred linear statistic.
\paragraph{Definition of the threshold.} The next proposition provides
bounds for the quantiles of the test statistic $\tSd$. Recall that $L_x = (1
- \log x)^{1/2}>1$ for $x \in (0,1)$.  For $k \in \IN$
and $\alpha \in (0,1)$ with
$c_1:= 799 \Vnorm[{\lp[2]}]{\fydfO}+ 1372, c_2 := 52  \Vnorm[{\lp[1]}]{\fydfO}$ we define
the threshold
\begin{equation}
	\label{d:tau}
        \tSdC := c_1  \lb 1 \vee
  L_\alpha^2\frac{(2k)^{1/4}}{n^{1/2}} \vee L_\alpha^{3}\frac{(2k)^{1/2}}{n} \rb
  {L_\alpha} \frac{(2k)^{1/2}}{n} + c_2 L^2_\alpha \frac{1}{n}.
\end{equation}
\begin{proposition}\label{d:quantiles}
  For 
  $\xdfO,\xdf,\edf\in\Lp[2]$ and $n\in\Nz$, $n\geq2$, consider $\lcb \yOb[j] \rcb_{j=1}^n \iid \ydf=\xdf\oast\edf$
  with joint distribution $\FuVg{\xdf}$ and
  $\ydfO=\xdfO\oast\edf$. Let $\alpha, \beta \in (0,1)$ and for
  $k\in\Nz$ consider $\tSd$ and $\tSdC$ as in
  \eqref{tqfk} and \eqref{d:tau}, respectively.
  \begin{resListeN}
  \item\label{d:alpha:quantile}If   $\gLp[\ydfO]=\RLp$, then 
    $	\Pf[\xdfO] \lb \tSd \geq \tSdC  \rb \leq \alpha.$
  \item\label{d:beta:1} If
    $c_3:=837\Vnorm[{\lp[2]}]{\fedf} +1373$ and
    \begin{equation}
      \label{d:separation}
      \qF_k(\ydf - \ydfO) \geq  2\big( \tSdC  +    c_3L_{\beta/2}^4 \big( 1  \vee
      \frac{(2k)^{1/2}}{n}\big) \frac{(2k)^{1/2}}{n}\big)                    ,
    \end{equation}	
    then	$ \Pf[\xden] \lb \tSd < \tSdC \rb \leq \beta.$
  \end{resListeN}    
\end{proposition}
\begin{proof}[Proof of \cref{d:quantiles}]
  The result \ref{d:alpha:quantile} may be proved in much the same way
  as \cref{quantiles} \ref{alpha:quantile} using the decomposition
  \eqref{d:decomposition} rather than \eqref{decomposition} and
  applying \cref{simplified:con} in the appendix, a concentration
  inequality for canonical U-statistics of order 2, together with
  \cref{d:prop:u} rather than \cref{prop:u}.  Secondly, consider
  \ref{d:beta:1}. Keeping the decomposition \eqref{d:decomposition} in
  mind we control the deviations of the U-statistic $\uSd$ and the
  linear statistic $\lSd$ applying \cref{simplified:con} and
  \cref{d:linear} in the appendix, respectively.  In fact, the proof
  is similar to the proof of \cref{quantiles} \ref{beta:1} using
  \cref{d:prop:u} and \cref{d:linear} rather than \cref{prop:u} and
  \cref{linear}, and we omit the details.
\end{proof}
\paragraph{Definition of the test.} For $k\in\Nz$ and $\alpha\in(0,1)$
using the test statistic $\tSd$
and the threshold $\tSdC$ given in \eqref{d:decomposition} and
\eqref{d:tau}, respectively, we consider the test
\begin{align*}
	\tFd:= \Ind{\{\tSd \geq \tSdC\}}. 
\end{align*}
From \ref{d:alpha:quantile} in \cref{d:quantiles} it immediately
follows that $\tFd$ is a level-$\alpha$-test for all $k \in
\IN$. Moreover, \cref{d:quantiles} \ref{d:beta:1} allows to characterises 
elements in $\rwC$ (defined in \eqref{ellipsoid}) for which
$\tFd$ is powerful. Exploiting these results, in the next proposition
we derive an upper bound for the radius of testing of $\tFd$ in terms
of $m_k^2= \max_{|j|\in\nset{k}} |\fedf[j]|^{-2}$ as  in \eqref{nu}
and the regularity parameter $\wC\in\altcollect$. Thus, we define
\begin{align*}
  \tSRdK{k}:= \tSRdKN{k}{n}  := \wdclass[k]^2 \vee
  \frac{(2k)^{1/2}}{n} m_k^2  .
\end{align*}
\begin{proposition}\label{d:re:ub}
  Under the assumptions of \cref{d:quantiles}  
  for  $\alpha \in (0,1)$ define
  \begin{equation}
\oA_\alpha^2 := R^2 +  2 (837\Vnorm[{\lp[2]}]{\fedf}+
851\Vnorm[{\lp[1]}]{\fydfO}+2745  ) L_{\alpha/4}^4.\label{d:re:ub:A}
\end{equation}
For all  $A \geq \oA_\alpha$ and for all  $n, k \in \IN$ with 
$2k\leq n^2$ we have 
	$\mc R(\tFd[k,\alpha/2] \mid \rwC, A \tRdK{k}) \leq \alpha$.
\end{proposition}
\begin{proof}[Proof of \cref{d:re:ub}]
 Due to \cref{d:quantiles}  both the type I and the
  maximal type II error probability are bounded
  by $\alpha/2$, and thus the result follows immediately from the
definition of the risk 
  $\trisk[{\tFd[k,\alpha/2]} \mid \rwC, A \tRdK{k}]$. 
  Indeed, since the assumption of \cref{d:quantiles}
  \ref{d:alpha:quantile} is fulfilled the test $\tFd[k,\alpha/2]$ is a
  level-$\alpha/2$-test. Hence, in order to apply \cref{d:quantiles}
  \ref{d:beta:1} (with $\beta = \alpha/2$) for each density
  $\xdf\in\Lp[2]$ with
  $\Vnorm[{\Lp[2]}]{\xdf-\xdfO}^2 \geq \oA_\alpha^2 \tSRdK{k}$ and
  $\xdf - \xdfO \in \rwC$ it remains to verify condition
  \eqref{d:separation} which states
  $ \qF_k(\ydf - \ydfO) \geq 2\big( \tSdCa{{\alpha}/{2}} + c_3
  L_{\alpha/4}^4 ( 1 \vee (2k)^{1/2}n^{-1}) (2k)^{1/2}n^{-1}\big) $ with
  $\tSdCa{{\alpha}/{2}} $ as in \eqref{d:tau}.
  Since
  $ \sum_{\lv j \rv > k} |\fxdf[j] - \fxdfO[j]|^2 \leq
  \wdclass[k]^2 \rdclass^2$ it follows   
  \begin{multline}\label{d:re:ub:bew:e1}
    \qF_k(\xden -\xoden)  = \Vnorm[{\Lp[2]}]{\xdf-\xdfO}^2-
    \sum_{\lv j \rv > k} |\fxdf[j] - \fxdfO[j]|^2
    \geq \oA_\alpha^2 \tSRdK{k} - \wdclass[k]^2 \rdclass^2 \\
     \geq 2(837\Vnorm[{\lp[2]}]{\fedf}+799\Vnorm[{\lp[2]}]{\fydfO}+52\Vnorm[{\lp[1]}]{\fydfO}+2745) L_{\alpha/4}^4 \frac{(2k)^{1/2}}{n}m_k^2
  \end{multline}
  by using $\Vnorm[{\lp[2]}]{\fydfO}\leq\Vnorm[{\lp[1]}]{\fydfO}$.
  The  condition \eqref{d:separation} 
   follows from \eqref{d:re:ub:bew:e1}
  together with $\qF_{k}(\xdf-\xdfO)\leq m_k^2
  \qF_{k}(\ydf-\ydfO)$ 
  by exploiting further  $1\leq L_{\alpha/2}\leq L_{\alpha/4}$ and
  $2k\leq n^2$, which completes the proof. 
\end{proof}
The upper bound $ \tSRdK{k}$ for the radius of testing of
$\tFd[k,\alpha/2]$ depends on the dimension parameter $k$. Let us
introduce a dimension that  realises an optimal bias-variance
trade-off, and the  corresponding radius
\begin{multline}\label{tDd}
	\tDd := \argmin_{k \in \IN}  \tRdK{k}  := \min \lcb k \in \IN: \tRdK{k} \leq  \tRdK{l}\text{ for all } l \in \IN \rcb\quad\text{ and}\\
  \tSRd:=  \min_{k \in \IN}\tSRdK{k}= \min_{k \in \IN} \{ a_k^2 \vee
  \frac{(2k)^{1/2}}{n} m_k^2 \}.
\end{multline}
\begin{corollary}\label{d:re:m:ub}
  Under the assumptions of \cref{d:quantiles}  
 let  $\alpha \in (0,1)$ and $\oA_\alpha$ as in \eqref{d:re:ub:A}, 
 then $\mc R(\tFd[\tDd,\alpha/2] \mid \rwC, A \tRd) \leq \alpha$ 
 for all $A \geq \oA_\alpha$ and for all $n \geq \sqrt{2} |\fedf[1]|^{-2}$.
\end{corollary}
\begin{proof}[Proof of \cref{d:re:m:ub} ]The result follows immediately
  from \cref{d:re:ub}, since $2\tDd\leq n^2$ for all $n \geq
  \sqrt{2} \lv \epsden[1] \rv^{-2}$. Indeed, $n \geq \sqrt{2}
  \lv \epsden[1] \rv^{-2}$ implies $1\geq \tSRdK{1}\geq \tSRd\geq
  (2\tDd)^{1/2} n^{-1}$.
\end{proof}

\begin{remark}[Optimality of the direct testing procedure]
	Let us compare the upper bound $\tSRd = \min_{k \in \IN} \lcb a_k^2 \vee  (2k)^{1/2}m_k^2n^{-1} \rcb$ for the direct testing procedure with the minimax radius of testing $\tSRi = \min_{k \in \IN} \lcb a_k^2 \vee \nu_k^2n^{-1} \rcb$. If there exists a constant $c > 0$ such that
	\begin{align}
		\label{cond:direct:optimal}
          \nu_k^2=\big(\sum_{|j|\in\nset{k}} \lv \epsden[j]
          \rv^{-4}\big)^{1/2}\leq (\max_{|j|\in\nset{k}} |\fedf[j]|^{-2})(2k)^{1/2}=m_k^2(2k)^{1/2}\leq   c \nu_k^2,
	\end{align}
	then $\tRd$ and $\tRi$ are of the same order
        and, thus, the direct testing procedure is minimax
        optimal. Condition \eqref{cond:direct:optimal} is for instance
        satisfied for a mildly ill-posed model, 
        i.e. if  $\Nsuite[j]{|\fedf[j]|}$ decays  polynomially. Note,
        however, that \eqref{cond:direct:optimal} is a sufficient but
        not a necessary condition. For a severely ill-posed model,
        i.e. if $\Nsuite[j]{|\fedf[j]|}$ decays exponentially,
        \eqref{cond:direct:optimal} is not fulfilled. Nevertheless,
        the direct testing procedure still performs optimally (see  \cref{d:ill} below). 
\end{remark}

\begin{illustration}
  \label{d:ill}
  We illustrate the order of the upper bound for the radius of testing of the direct test $\tFd[\tDd,\alpha/2]$ 
        under the
	regularity and ill-posedness assumptions introduced in \cref{ill}. Comparing the resulting upper bounds $\tSRd$ with the  radii $\tSRi$, 
        we conclude that the direct test performs as well as the
        indirect test  in all three cases.\\ [3ex]
	  			  			\centerline{\begin{tabular}{ll|l|l}\toprule
			\multicolumn{4}{c}{Order of the optimal dimension $	\tDd$ and the upper bound $\tSRd$}\\ \midrule
			${\wC[j]}$ & ${|\fedf[j]|}$ & $	\tDd$ & $\tSRd$   \\
			(smoothness) & (ill-posedness) & & \\
			\midrule
			${ j^{-\sPara}}$ & ${\lv j \rv^{-\pPara}}$ &  $n^{\tfrac{2}{4\pPara + 4\sPara + 1}}$ & $\say^{-\tfrac{4\sPara}{4\sPara + 4 \pPara + 1}}$  \\
			${ j^{-\sPara}}$ & 	$ {e^{-\lv j \rv^{\pPara}}}$ & $(\log n)^{\tfrac{1}{\pPara}}$ & $(\log \say)^{-\tfrac{2\sPara}{\pPara}}$  \\
			${e^{-j^{\sPara}}}$ & $ {\lv j \rv^{-\pPara}}$& $ (\log n)^{\tfrac{1}{\sPara}}$ & $\say^{-1} (\log n)^{\tfrac{2 \pPara + 1/2}{\sPara}}$   \\
			\bottomrule
            \end{tabular}}\\[3ex]
   In a severely ill-posed model with ordinary smoothness,  the order of the optimal dimension parameter does not depend on the smoothness parameter. Hence, the test $\tFd[\tDd,\alpha/2]$ is automatically adaptive with respect to ordinary smoothness and no aggregation procedure is needed.
\end{illustration}

\section{Adaptive direct testing procedure}\label{sec:a:u:b:d}
\paragraph{Adaptation procedure via Bonferroni aggregation.}  The
choice of the optimal dimension parameter used in the test
$\tFd[\tDd,\alpha/2]$ in \cref{d:re:m:ub} requires the knowledge of the
regularity parameter $\wC$ belonging  to a set  $\altcollect$ 
of strictly positive, monotonically non-increasing sequences bounded by $1$.
Let $\cK \subseteq \IN$ be a finite
collection of dimension parameters. Applying the Bonferroni
aggregation method described in \cref{sec:a:u:b} to the collection of
direct tests $(\tFd[k,\alpha/\lv \mc K \rv])_{k \in \mc K}$
we introduce the direct $\max$-test with Bonferroni choice of the error level, 
i.e. 
\begin{align*}
	\tFda := \Ind{\{ \tSdM > 0\}} \qquad \text{ with } \qquad \tSdM: = \max_{k \in \mc K} \lb \tSd - \tSdCa{\tfrac{\alpha}{|\cK|}}\rb.
\end{align*}
\paragraph{Testing radius of the indirect $\max$-test.}
We define the minimal achievable radius of testing $\tRdaN{n}$ over the set $\mc K$  in terms of $m_k$ as in \eqref{nu} and the regularity parameter $\wdclass[\mbullet] \in \altcollect$ by
\begin{align*}
\tRdaN{x} := \min_{k \in \mc K}\tRdKN{k}{x} \quad\text{ with
}\quad\tSRdKN{k}{x}:=  a_k^2 \vee (2k)^{1/2}m_k^2x^{-1} \quad\text{ for all }x\in\pRz. 
\end{align*}
Since $\tRdN{n}=\tRdKN{\Nz\,}{n}$  in \eqref{tDd} is defined as  minimum taken over $\IN$
instead of $\mc K$, for all $n\in\Nz$ we always have $\tRdaN{n} \geq \tRdN{n}$. Let us furthermore recall the remainder radius $\tSReaN{x}=\min_{k \in \mc K} \{a_k^2 \vee m_k^2x^{-1}\}$ defined in
\eqref{tSReaN} for $x\in\pRz$.
\begin{proposition}[Uniform radius of testing over $\altcollect$]
  \label{d:adapt:ub}
  Under the assumptions of \cref{d:quantiles}  
  let  $\alpha \in (0,1)$ and consider $\oA_\alpha$ as in \eqref{re:ub:A}, 
    Then, for all $A \geq \oA_\alpha$ and for all $n\in\Nz$
  \begin{align*}
    \sup_{\wC \in \altcollect} \trisk[{\tFd[\cK,\alpha/2]} \mid \rwC, A(\tReaN{\delta^2n}\vee \tRdaN{\delta n})(1\vee\delta^{-3/2}\tRdaN{\delta n}) ]  \leq \alpha
	\end{align*}
	with	$\delta = (1+\log |\cK|)^{-1/2}$.
\end{proposition}
\begin{proof}[Proof of \cref{d:adapt:ub}]
  The proof follows along the lines of the proof of \cref{adapt:ub}
  making use of \cref{d:quantiles} rather than
  \cref{quantiles}. Similarly to
  \eqref{adapt:null} together with \cref{d:quantiles}
  \ref{d:alpha:quantile} and
  $\sum_{k \in \mc K} \frac{\alpha}{2|\cK|} = \frac{\alpha}{2}$ it follows
  that the type I
  error probability is bounded by $\alpha/2$.   Under the
  alternative, let $\xdf\in\Lp[2]$ with $\xdf-\xdfO\in
  \rwC$ satisfy
  \begin{equation}
    \label{d:adapt:ub:bew:e1}
    \Vnorm[{\Lp[2]}]{\xdf-\xdfO}^2 \geq \oA_\alpha^2(\tSReaN{\delta^2n}\vee \tSRdaN{\delta n})(1\vee\delta^{-3}\tSRdaN{\delta n}).
  \end{equation}
  It is sufficient to use the elementary  bound
  \eqref{adapt:alternative} together with the observation that 
  \begin{resListeN}
  \item\label{d:adapt:ub:bew:c1}
    $ \FuVg{\xdf}( \tSd < \tSdCa{\tfrac{\alpha}{2|\cK|}})\leq
\frac{\alpha}{2}$ holds for all $k\in\cK$ whenever $\xdf\in\Lp[2]$
    satisfies $\xdf-\xdfO\in \rwC$ and
  \begin{equation}
    \label{d:adapt:ub:bew:c1:e}
    \Vnorm[{\Lp[2]}]{\xdf-\xdfO}^2 \geq \oA_\alpha^2 \lb \wC[k]^2\vee(1 \vee
  \frac{(2k)^{1/4}}{\delta^2n^{1/2}} \vee \frac{(2k)^{1/2}}{\delta^3n}
  ) \frac{(2k)^{1/2}m^2_k}{\delta n} \vee \frac{m^2_k}{\delta^2n}\rb;
  \end{equation}
\item \label{d:adapt:ub:bew:c2} under \eqref{d:adapt:ub:bew:e1}  there exists $k\in\cK$ such that also  \eqref{d:adapt:ub:bew:c1:e}
  is fulfilled. 
\end{resListeN}    
Consequently,  we
have $\FuVg{\xdf}(\tFda[\alpha/2, \cK] = 0)\leq \alpha/2$ for all
$\xdf\in\Lp[2]$ satisfying  $\xdf-\xdfO\in
  \rwC$ and \eqref{d:adapt:ub:bew:e1}. Since  both the type I and the
  maximal type II error probability are bounded
  by $\alpha/2$ \cref{d:adapt:ub} follows immediately from the
definition of the risk.  It remains to show \ref{d:adapt:ub:bew:c1} and
\ref{d:adapt:ub:bew:c2}. The claim \ref{d:adapt:ub:bew:c1} follows from
\cref{d:quantiles} \ref{beta:1} (with  $\beta = \alpha/2$)  since
\eqref{d:adapt:ub:bew:c1:e} implies the
  condition \eqref{d:separation},   which  states      $ \qF_k(\ydf - \ydfO) \geq  2\big( \tSdCa{\tfrac{\alpha}{2|\cK|}}  +
      c_3 L_{\alpha/4}^4 ( 1  \vee
      (2k)^{1/2}n^{-1}) (2k)^{1/2}n^{-1}\big)              $	
    with $\tSdCa{\tfrac{\alpha}{2|\cK|}} $ as in
    \eqref{d:tau}.  Indeed,  
we have 
\begin{equation*}
\tSdCa{\tfrac{\alpha}{2|\cK|}}
  \leq (c_1+c_2)L_{\alpha/2}^4  \lb (1 \vee
  \frac{(2k)^{1/4}}{\delta^2n^{1/2}} \vee \frac{(2k)^{1/2}}{\delta^3n}
  ) \frac{(2k)^{1/2}}{\delta n} \vee \frac{1}{\delta^2n}\rb.
\end{equation*}
Using 
$\oA_\alpha^2-\rdclass^2\geq
2(c_1+c_2+c_3)L_{\alpha/4}^4$ elementary calculations show that \eqref{d:separation} holds whenever
\begin{equation}\label{d:cond}
  \qF_k(\ydf - \ydfO) \geq  \big(\oA_\alpha^2-\rdclass^2\big)
  \lb (1 \vee
  \frac{(2k)^{1/4}}{\delta^2n^{1/2}} \vee \frac{(2k)^{1/2}}{\delta^3n}
  ) \frac{(2k)^{1/2}}{\delta n} \vee \frac{1}{\delta^2n}\rb.
\end{equation}	
Due to  $\xdf-\xdfO\in
  \rwC$ and hence $ \sum_{\lv j \rv > k}  |\fxdf[j] - \fxdfO[j]|^2 \leq
  \wdclass[k]^2 \rdclass^2$,  the condition \eqref{d:adapt:ub:bew:c1:e} implies
 \begin{multline*}
   \qF_k(\xdf -\xdfO)  = \Vnorm[{\Lp[2]}]{\xdf-\xdfO}^2-
    \sum_{\lv j \rv > k} |\fxdf[j] - \fxdfO[j]|^2\\\geq
    m_k^2\big(\oA_\alpha^2-\rdclass^2\big)
  \lb (1 \vee
  \frac{(2k)^{1/4}}{\delta^2n^{1/2}} \vee \frac{(2k)^{1/2}}{\delta^3n}
  ) \frac{(2k)^{1/2}}{\delta n} \vee \frac{1}{\delta^2n}\rb.
\end{multline*}
As a consequence, if \eqref{d:adapt:ub:bew:c1:e} is satisfied then due
to    $m_k^2\qF_{k}(\ydf-\ydfO)\geq \qF_k(\xdf -\xdfO)$  also
\eqref{d:cond}, and thus \eqref{d:separation}, which shows the claim
\ref{adapt:ub:bew:c1}. Lastly, consider \ref{d:adapt:ub:bew:c2}. By Lemma A.1 in
\cite{SchluttenhoferJohannes2020} we have  $\tSReaN{\delta^2n}\vee \tSRdaN{\delta
  n}
=\wC[k]\vee
\tfrac{m_k^2}{\delta^{2}n}\vee \tfrac{(2k)^{1/2}m_k^2}{\delta n}$ and $\tSRdaN{\delta
  n}\geq (2k)^{1/2}\delta^{-1}n^{-1}$ for  at least one $k \in
\cK$. 
Hence, there is $k\in\cK$ with $\big(\tSReaN{\delta^2n}\vee\tSRdaN{\delta n}\big)(1\vee\tfrac{\tSRdaN{\delta
      n}}{\delta^{3}})\geq
  \wC[k]^2\vee\tfrac{m_k^2}{\delta^{2}n}\vee\big( \tfrac{(2k)^{1/2}m_k^2}{\delta
    n}(1\vee \tfrac{(2k)^{1/2}}{\delta^{4}n})\big)$. Since $1 \vee
 \tfrac{(2k)^{1/2}}{\delta^4n}\geq1 \vee
  \frac{(2k)^{1/4}}{\delta^2n^{1/2}} \vee \frac{(2k)^{1/2}}{\delta^3n}
  $, this shows
  \ref{d:adapt:ub:bew:c2} and completes the proof.
\end{proof}

\begin{corollary}[Worst-case adaptive factor]
  \label{d:cor:worst}  Under the assumptions of \cref{d:quantiles}  
  let  $\alpha \in (0,1)$ and consider $\oA_\alpha$ as in \eqref{d:re:ub:A}. 
  Then   for
  all $A \geq \oA_\alpha$ and $n\in\Nz$  
  \begin{equation*}
    \sup_{\wC \in \altcollect} \trisk[{\tFda[\cK,\alpha/2]}  \mid \rwC, A \tRdaN{\delta^2 n}(1\vee\tRdaN{\delta^2 n}) ]  \leq \alpha.
  \end{equation*}
  with  $\delta = (1+\log \lv \mc K \rv)^{-1/2} $.
\end{corollary}
\begin{proof}[Proof of \cref{d:cor:worst}]The result may be proved in
  much the same way as  \cref{d:adapt:ub}.    In fact,  as in the proof of \cref{d:adapt:ub}, it is sufficient to show
  \ref{d:adapt:ub:bew:c2} replacing  \eqref{d:adapt:ub:bew:e1} by
    \begin{equation}
    \label{d:cor:worst:bew:e1}
    \Vnorm[{\Lp[2]}]{\xdf-\xdfO}^2 \geq \oA_\alpha^2\tSRdaN{\delta^2
      n}(1\vee\tSRdaN{\delta^2 n}).
  \end{equation}
 For each
 $\wC\in\altcollect$  under  \eqref{d:cor:worst:bew:e1}  
  the dimension parameter  $k:=\argmin\nolimits_{k' \in \mc
   K}\tSRdKN{k'}{\delta^2 n}$ satisfies
  \begin{multline*}
    \tSRdaN{\delta^2 n}(1\vee\tSRdaN{\delta^2 n})=
    \tSRdaN{\delta^2 n}(1\vee\tRdaN{\delta^2
      n}\vee\tSRdaN{\delta^2 n})\\\geq
    \wC[k]^2\vee\tfrac{m_k^2(2k)^{1/2}}{\delta^{2}n}(1\vee\tfrac{(2k)^{1/4}}{\delta
      n^{1/2}} \vee\tfrac{(2k)^{1/2}}{\delta^{2}n})\geq
    \wC[k]^2\vee\tfrac{m_k^2}{\delta^{2}n}\vee\tfrac{m_k^2(2k)^{1/2}}{\delta
      n}(1\vee\tfrac{(2k)^{1/4}}{\delta^2 n^{1/2}}
    \vee\tfrac{(2k)^{1/2}}{\delta^{3}n})
  \end{multline*}
  since
  $\wC[k]^2\vee\tfrac{m_k^2(2k)^{1/2}}{\delta^{2}n}=\tSRdaN{\delta^2
    n}$ and $\delta\leq1$. This shows \eqref{d:adapt:ub:bew:c1:e}  and, in consequence, \ref{d:adapt:ub:bew:c2}. We
  obtain the assertion  proceeding exactly as
  in the proof of \cref{d:adapt:ub}. 
\end{proof}
By
\cref{d:cor:worst}, $\tSRdaN{\delta^2 n}$ is an upper bound
for the radius of testing of the  direct $\max$-test if $\tSRdaN{\delta^2 n}\leq 1$. The latter   is satisfied
for an arbitrary regularity parameter $\wC\in\altcollect$, if $1
\in  \mc K$ and $ n \geq \sqrt{2} \lv \epsden[1] \rv^{-2}(
1+\log|\cK|) $.  The next corollary establishes $\tSRdaN{\delta n}$ as
a sharper upper bound for the radius of testing of the  direct $\max$-test
under additional conditions, which are satisfied in all the examples considered in \cref{d:ill:adapt}
below.
The result follows
immediately from \cref{d:adapt:ub} and we omit its proof. 

\begin{corollary}[Best-case adaptive factor] \label{b:c:a:f}
  Under the assumptions of \cref{d:quantiles} let $\alpha \in (0,1)$ and
  consider
  $\oA_\alpha$ as in \eqref{d:re:ub:A}. If there exists a constant $C\geq1$
  such that $\tReaN{\delta^2n}\leq C \tRdaN{\delta n}$ and
  $\tRdaN{\delta n}\leq C \delta^{3/2}$ for all
  $\wC \in \altcollect$, then for all
  $A \geq C^2 \overline{A}_{\alpha}$ and $n\in\Nz$
  \begin{equation*}
    \sup_{\wC \in \altcollect} \trisk[{\tFda[\cK,\alpha/2]}  \mid \rwC, A\tRdaN{\delta n} ]  \leq \alpha
  \end{equation*}
   with
  $\delta = (1+\log \lv \mc K \rv)^{-1/2} $.
\end{corollary}

Concerning the choice of the collection $\cK$ of dimensions we refer
to \cref{choice:of:K}
\begin{illustration}
  \label{d:ill:adapt} 
  For the typical configurations for regularity and ill-posedness
  introduced in \cref{ill} the tables below display the adaptive radii
  of the direct $\max$-test $\tFda[\cK,\alpha/2]$ for appropriately
  chosen grids. In a mildly ill-posed model with ordinary smoothness
  we choose the geometric grid $\cK_g$ and hence consider
  $\delta=(1+\log |\cK_g|)^{-1/2}$. It is easily seen that the
  remainder term $\tSReacKN{\cK_g}{\delta^2 n}$ is asymptotically
  negligible compared with $\tSRdKN{\cK_g}{\delta n} $. Moreover,
  $\delta^{-3}\tSRdKN{\cK_g}{\delta n}$ tends to zero. Hence, the
  upper bound in \cref{d:adapt:ub} asymptotically reduces to
  $\tSRdKN{\cK_g}{\delta n}$ featuring an adaptive factor of order
  $(\log\log n)^{1/2}$ as can be seen in the table below. In a
  severely ill-posed model we have seen in \cref{d:ill} that a direct
  test with optimal dimension is automatically adaptive with respect
  to ordinary smoothness. However, the optimal dimension depends on
  the ill-posedness parameter. Therefore, we consider here the
  $\max$-test $\tFd[\cK_g,\alpha/2]$ also in this situation.  Since
  $\delta^{-3}\tSRdKN{\cK_g}{\delta n}$ tends to zero, and both the
  remainder term $\tSReacKN{\cK_g}{\delta^2 n}$ and
  $\tSRdKN{\cK_g}{\delta n}$ are of the same order as
  $\tSRdKN{\cK_g}{n} $, the upper bound in \cref{d:adapt:ub}
  asymptotically reduces to $ \tSRdKN{\cK_g}{n} $.
  \\ [3ex]
  \centerline{%
    \begin{tabular}{ll|l|l}\toprule \multicolumn{4}{l}{Order
                of $\tSReacKN{\cK_g}{\delta^2 n}$ and
                $\tSRdKN{\cK_g}{\delta n}$ with $\delta=(1+\log
                |\cK_g|)^{-1/2}$}\\
                \multicolumn{4}{r}{and
                $\cK_g= \{1\}\cup\{2^j,j\in\nset{\lfloor
                \log_2({n^2}/{2})
                \rfloor}\}$} \\
                \midrule
                ${\wC[j]}$ & ${| \fedf[j]|}$ & $\tSReacKN{\cK_g}{\delta^2 n}$ &$\tSRdKN{\cK_g}{\delta n}$   \\
                (smoothness) & (ill-posedness) &  \\
                \midrule
                ${ j^{-\sPara}}$ & ${\lv j \rv^{-\pPara}}$  & $\lb \frac{\say}{{\log\log n}} \rb^{-\tfrac{4\sPara}{4\sPara + 4 \pPara}}$  & $\lb \frac{\say}{(\log\log n)^{1/2}} \rb^{-\tfrac{4\sPara}{4\sPara + 4 \pPara + 1}}$  \\
                ${ j^{-\sPara}}$ & $ {e^{-\lv j \rv^{\pPara}}}$  & $(\log \say)^{-\tfrac{2\sPara}{\pPara}}$  & $(\log \say)^{-\tfrac{2\sPara}{\pPara}}$  \\
                \bottomrule
    \end{tabular} }
  \\[3ex]	
  In a mildly ill-posed model with super smoothness, we choose the smaller geometric grid $\cK_{\sPara_\star}$ for adaptation to smoothness $\sPara \geq\sPara_\star$ and hence $\delta=(1+\log
  |\cK_{\sPara_\star}|)^{-1/2}$. It is easily seen that the remainder term
  $\tSReacKN{\cK_{\sPara_\star}}{\delta^2 n}$   is asymptotically negligible
  compared with $\tSRdKN{\cK_{\sPara_\star}}{\delta n} $ and  $\delta^{-3}\tSRdKN{\cK_{\sPara_\star}}{\delta n}$ tends to zero. Hence, the upper bound in \cref{d:adapt:ub}  asymptotically reduces to
  $\tSRdKN{\cK_{\sPara_\star}}{\delta n}$ featuring an adaptive factor of order
  $(\log\log\log n)^{1/2}$.\\ [3ex]
  \centerline{%
    \begin{tabular}{ll|l|l}\toprule
      \multicolumn{4}{l}{Order of $\tSReacKN{\cK_{\sPara_\star}}{\delta^2
      n}$ and $\tSRiKN{\cK_{\sPara_\star}}{\delta n}$ with  $\delta=(1+\log
      |\cK_{\sPara_\star}|)^{-1/2}$}\\
      \multicolumn{4}{r}{and  $\cK_{\sPara_\star}=
      \{1\}\cup\{2^j,j\in\nset{\lfloor {s^{-1}_{\star}} \log_2 \log n
      \rfloor}\}$} \\
      \midrule
      ${\wC[j]}$ & ${|\fedf[j]|}$ & $\tSReacKN{\cK_{\sPara_\star}}{\delta^2 n}$ & $\tSRiKN{\cK_{\sPara_\star}}{\delta n}$    \\
      (smoothness) & (ill-posedness) &  \\
      \midrule
      ${ e^{-j^\sPara}}$ & ${\lv j \rv^{-\pPara}}$  & $ \frac{{\log\log\log n}}{n} (\log n)^{\frac{2 \pPara}{\sPara}}$  &  $ \frac{(\log\log\log n)^{1/2}}{n}  (\log n)^{\frac{2 \pPara + 1/2}{\sPara}}$  \\
      \bottomrule
    \end{tabular}}
  \\[3ex]
  We conclude that in all the cases considered in this illustration the direct $\max$-test achieves a testing radius of the same order as the indirect $\max$-test. 
\end{illustration}
We emphasise that in the case $\xdfO = \mathds{1}_{[0,1)}$ the direct
$\max$-test in contrast to the indirect $\max$-test does not require
any knowledge about the error density $\edf$. Indeed, neither the test
statistic $\tSd$ in \eqref{tqfk} nor the threshold $ \tSdC $ in
\eqref{d:tau} depends on characteristics of the error density $\edf$.
However, in a mildly ill-posed model with ordinary smoothness both the
direct and the indirect $\max$-test feature an adaptive factor of
order $(\log\log n)^{1/2}$ which we show below is unavoidable when
testing for uniformity. Moreover, considering only super smooth
densities in a mildly ill-posed model both the direct and the indirect
$\max$-test share an adaptive factor of order
$(\log\log\log n)^{1/2}$ which we show below also is unavoidable when
testing for uniformity. It is important to note that without any
prior knowledge about the ill-posedness of the model the direct
$\max$-test for uniformity attains the optimal testing radius
simultaneously in a mildly and severely ill-posed model with ordinary
smoothness.

\section{Lower bound}\label{sec:l:b}
Throughout this section we consider testing for uniformity, i.e.  $\xdfO 
=\mathds{1}_{[0,1)}$. The next proposition states general conditions on the class
$\altcollect$ under which an adaptive factor $\delta^{-1}$ is an
unavoidable cost to pay for adaptation over $\altcollect$. The proof
of \cref{prop:adaptive:lb} makes use of \cref{adapt:chi2}  in the
appendix, which provides  a bound on the $\chi^2$ divergence between
the null and a mixture over several alternative classes.  
Inspired by Assouad's cube technique the candidate densities, i.e. the
vertices of  the hypercubes, are constructed such that, roughly speaking,  they are statistically indistinguishable from the null $\xdfO$ while having largest possible $\Lp[2]$-distance. 
\begin{proposition}[\textbf{Adaptive lower bound}]
  \label{prop:adaptive:lb}
  Let $\alpha \in (0,1)$ and $\delta \in (0,1]$. Assume a 
   collection of $N\in\Nz$ regularity parameters $\{ \mwC{m}\in \altcollect: m
   \in\nset{N} \}$, where we abbreviate
   $\mtRi:=\tRiN[\mwC{m}]{ \delta n}$ with associated $\mDi := \tDi[\mwC{m}]$ for
   $m \in\nset{N}$ as in \eqref{tDi}, satisfies the following four conditions:
   \begin{resListeN}[\setListe{-0.5ex}{1ex}{4ex}{0ex}\renewcommand{\theListeN}{\small\normalfont\rmfamily\dgrau(C\arabic{ListeN})}]
   \item \label{(C1)}
     $\mDi\leq\mDi[l]$ and $ \mtRi \leq \delta \mtRi[l]$  whenever $m<l$ and  $l,m\in\nset{N}$,
   \item\label{(C2)}there is a finite constant $c_\alpha > 0$ such that $\exp(c_\alpha   \delta^{-2}) \leq N \alpha^2$,
    \item\label{(C3)} there is a  finite constant $\ssconst > 0$ such that
      $2\max_{m\in\nset{N}}\Vnorm[{\lp[2](\Nz)}]{\mwC{m}}^2\leq \ssconst$,
    \item\label{(C4)} there is a constant $\eta \in (0,1]$ such that 
  \begin{align}
    \label{eta}
    \eta \leq \min_{m \in \nset{N}}
    \frac{(\mwC[\mDi]{m})^2\wedge
    (\delta n)^{-1}\nu_{\mDi}^2}{(\mwC[\mDi]{m})^2\vee
    (\delta n)^{-1}\nu_{\mDi}^2} =  \min_{m \in \nset{N}}
    \frac{(\mwC[\mDi]{m})^2\wedge
    (\delta n)^{-1}\nu_{\mDi}^2}{(\mtRi)^2}.
  \end{align}
    \end{resListeN}
  Then, with $\uA_\alpha^2 := \eta \lb \rdclass^2 \wedge \sqrt{  \log(1+ \alpha^2)} \wedge {\ssconst}^{-1} \wedge \sqrt{{c_\alpha}} \rb$, we obtain for all $A \in [0,\uA_\alpha]$
  \begin{align}
    \label{eq:adapt:lb}
    \inf_{\test} \sup_{\wC \in \altcollect}	\trisk[ \test \mid
    \rwC, A {\tRiN[\wC]{\delta n}}] \geq 1 - \alpha.
  \end{align}
\end{proposition}

\begin{remark}[\textbf{Conditions of \cref{prop:adaptive:lb}}]
Let us briefly discuss the conditions of
\cref{prop:adaptive:lb}. Under \ref{(C1)}  the collection of
regularity parameters $\altcollect$ is rich enough to make a factor $\delta^{-1}$ for adaptation unavoidable, i.e. it contains distinguishable elements  resulting in
significantly different radii. \ref{(C2)} is a bound for the maximal
size of an unavoidable adaptive factor. \ref{(C3)} guarantees that
the candidates  constructed in the reduction scheme of the proof are
indeed densities. The condition \ref{(C4)} relates the  behaviour of
the sequences $\fedf$ and $\mwC{m}$, $m\in\nset{N}$, and essentially guarantees an
optimal balance of the bias and the variance term in the dimension
$\mDi$ uniformly over $m\in\nset{N}$. Moreover, for all regularity and
ill-posedness examples considered in \cref{ill} \ref{(C4)}  holds
uniformly for all $n \in \IN$. We shall  emphasise that the optimal
dimensions $\mDi$ and the corresponding radii $\mtRi$ are
determined in terms of the effective sample size $\delta\say$. 
\end{remark}

\begin{proof}[Proof of \cref{prop:adaptive:lb}]
  \textbf{Reduction Step.} 
  Let $\FuVg{1}:=N^{-1}\sum_{m\in\nset{N}}\FuVg{\mwC{m}}$ with mixture $\FuVg{\mwC{m}}$ over the alternative $\rmwC\cap \Lp[2]_{\uA_\alpha\mtRi}$
  to be
  specified below, and set $\FuVg{0} :=
  \FuVg{\xdfO}$. Introducing the $\chi^2$-divergence $\chisq(\FuVg{0},\FuVg{1})$
  the risk is lower bounded
  due to a classical reduction argument as follows
  \begin{multline*}
    \inf_{\test} \sup_{\wC \in \altcollect}	\trisk[\test \mid
    \rwC, \uA_\alpha\tRiN{\delta n}]\geq \inf_{\test} \max_{m\in\nset{N}}\lb \FuVg{0}(\test = 1) + \FuVg{\mwC{m}}(\test = 0) \rb \\
    \geq \inf_{\test} \lb \FuVg{0}(\test = 1) + \FuVg{1} (\test = 0) \rb 
    \geq 1 - \sqrt{\frac{\chisq(\FuVg{1},\FuVg{0})}{2}}.
  \end{multline*}
  The last inequality follows e.g. from Lemma 2.5 combined with (2.7) in \cite{Tsybakov2009}. \\
  \textbf{Definition of the mixture.} For $m\in\nset{N}$ define
  $\mPa{m}\in\lp[2](\Nz)$ with
  \begin{equation}\label{re:a:lb:bew:0}
    \mPa[j]{m} := \uA_\alpha \mtRi \nu_{\mDi}^{-2} |\fedf[j]|^{-2}\text{ for
      all }j\in\nset{\mDi}\quad\text{ and }\quad\mPa[j]{m} :=0\text{ otherwise.}
  \end{equation}
  We first
  collect elementary properties of the $\lp[2](\Nz)$-sequences $\mPa{l}$,
  $\mPa{l}/\mwC{l}:=\Nsuite{\mPa[j]{l}/\mwC[j]{l}}$ and\linebreak
  $\mPa{l}\mPa{m}|\fedf|^2:=\Nsuite{\mPa[j]{l}\mPa[j]{m}|\fedf[j]|^2}$. Exploiting
  \eqref{eta}  shows by direct
  calculations%
  \begin{multline}\label{re:a:lb:bew:1}
     2\Vnorm[{\lp[2](\Nz)}]{\mPa{m}}^2=\uA_\alpha^2(\mtRi)^2, \qquad2\Vnorm[{\lp[2](\Nz)}]{\mPa{m}/\mwC{l}}^2\leq
  \uA_\alpha^2\frac{(\mtRi)^2}{\nu_{\mDi}^{4}}(\mwC[\mDi]{m})^{-2}\nu_{\mDi}^{4}\leq
  \uA_\alpha^2\eta^{-1} \mbox{ and }\\
n^22\Vnorm[{\lp[2](\Nz)}]{\mPa{m}\mPa{l}|\fedf|^2}^2=  \uA_\alpha^4\frac{(\mtRi)^2 }{\nu_{\mDi}^{4}}\frac{(\mtRi[l])^2 }{\nu_{\mDi[l]}^{4}}n^2(\nu_{\mDi}^{4}\wedge\nu_{\mDi[l]}^{4}).
  \end{multline}
   As a consequence,
  since $2\Vnorm[{\lp[2](\Nz)}]{\mwC{m}}^2\leq\ssconst$ by \ref{(C3)} and the
  definition of $\uA_\alpha^2$ we conclude that
  \begin{equation}\label{re:a:lb:bew:2}
    4\Vnorm[{\lp[1](\Nz)}]{\mPa{l}}^2\leq2\Vnorm[{\lp[2](\Nz)}]{\mPa{l}/\mwC{l}}^22\Vnorm[{\lp[2](\Nz)}]{\mwC{l}}^2\leq
    \eta^{-1}\uA_\alpha^2\ssconst\leq 1
    \mbox{ and  }
    2\Vnorm[{\lp[2](\Nz)}]{\mPa{l}/\mwC{l}}^2\leq
    \eta^{-1}\uA_\alpha^2\leq \rC^2.
  \end{equation}
  Let $m < l$, then $\nu_{\mDi}^{4}\leq\nu_{\mDi[l]}^{4}$ by
  \ref{(C1)}. Due to \eqref{re:a:lb:bew:1} combined with \eqref{eta} and \ref{(C1)} we have%
  \begin{equation}\label{re:a:lb:bew:3a}
    n^22\Vnorm[{\lp[2](\Nz)}]{\mPa{m}\mPa{l}|\fedf|^2}^2=\uA_\alpha^4\frac{(\mtRi)^2
    }{\delta^2 (\mtRi[l])^2}\frac{(\mtRi[l])^4 }{(\delta n)^{-2}\nu_{\mDi[l]}^{4}}\leq \uA_\alpha^4\eta^{-2}\frac{(\mtRi)^2
    }{\delta^2 (\mtRi[l])^2}\leq \uA_\alpha^4\eta^{-2}\leq \log(1+ \alpha^2).
  \end{equation}
  Let $l = m$, then from \eqref{re:a:lb:bew:1} exploiting \eqref{eta} we
  obtain%
  \begin{equation}\label{re:a:lb:bew:3b}
    n^22\Vnorm[{\lp[2](\Nz)}]{\mPa{l}\mPa{m}|\fedf|^2}^2=\uA_\alpha^4\frac{(\mtRi)^2
    }{\delta^2 (\mtRi[m])^2}\frac{(\mtRi[m])^4 }{(\delta
      n)^{-2}\nu_{\mDi[m]}^{4}}\leq
    \uA_\alpha^4\eta^{-2}\delta^{-2}\leq c_{\alpha}\delta^{-2}.
  \end{equation}
  Combining \eqref{re:a:lb:bew:3a} and \eqref{re:a:lb:bew:3b} from
  \ref{(C2)}   we  conclude 
  \begin{equation}\label{re:a:lb:bew:4}
    \tfrac{1}{N^2}\sum_{l,m\in\nset{N}}\exp\big(n^22\Vnorm[{\lp[2]}]{\mPa{l}\mPa{m}|\fedf|^2}^2\big)
    \leq \tfrac{1}{N}\exp\big(c_{\alpha}\delta^{-2}\big) +\tfrac{(N-1)}{N}
    (1+\alpha^2)\leq 2\alpha^2+1.
  \end{equation}  
  For each $m\in\nset{N}$, $\mPa{m}$ as in \eqref{re:a:lb:bew:0} and
  $\signv \in \{\pm\}^{\mDi}$, define a density
  \begin{equation*}
    \mxdf:=\expb[0]+\sum_{|j|\in\nset{\mDi}}\signv[|j|]\mPa[|j|]{m}\expb\in\cD.
  \end{equation*}
  Indeed, by construction $\mxdf$ belongs to $\Lp[2]$, integrates to
  one, is real-valued and positive, since
  $2\Vnorm[{\lp[1]}]{\mPa{m}}\leq1$ by \eqref{re:a:lb:bew:2}. Moreover,
  we have
  \begin{equation*}
    \mxdf-\xdfO=\sum_{|j|\in\nset{\mDi}}\signv[|j|]\mPa[|j|]{m}\expb\in\rmwC\cap \Lp[2]_{\uA_\alpha\mtRi}
  \end{equation*}
  exploiting
  $2\sum_{j\in\nset{\mDi}}(\mwC[j]{m})^{-2}|\signv[j]|^2|\mPa[j]{m}|^2=2\Vnorm[{\lp[2](\Nz)}]{\mPa{m}/\mwC{m}}^2\leq\rC^2$
  by \eqref{re:a:lb:bew:2} and
  $\Vnorm[{\Lp[2]}]{\mxdf-\xdfO}^2= 2\Vnorm[{\lp[2](\Nz)}]{\mPa{m}}^2=\uA_\alpha^2(\mtRi)^2$
  due to \eqref{re:a:lb:bew:1}. As a consequence
  $\FuVg{\mwC{m}}:= 2^{-\mDi} \sum_{\signv \in
    \{\pm \}^{\mDi}} \FuVg{\mxdf}$ is a  mixture on the
  alternative.\\ 
	 \textbf{Bound for the $\chisq$-divergence.} From \eqref{re:a:lb:bew:4}
	by applying \cref{adapt:chi2}  we   obtain 
	\begin{align*}
		\chisq(\FuVg{0},\FuVg{1}) \leq \frac{1}{N^2} \sum_{l,m
          \in\nset{N}}\exp\big(n^22\Vnorm[{\lp[2](\Nz)}]{\mPa{l}\mPa{m}|\fedf|^2}^2\big)
          - 1 \leq 2\alpha^2. 
	\end{align*}
        Combining this inequality with the reduction step yields the assertion \eqref{eq:adapt:lb}.
\end{proof}
\paragraph{Adaptive lower bounds in specific situations}
In the sequel we apply \cref{prop:adaptive:lb} to two specific classes of alternatives $\{\rwC : \wC \in \altcollect \}$. 
 We consider a set $\altcollect$ which is  non-trivial with respect to
 either
polynomial or exponential decay, that is, 
 $	\{(j^{-\sPara})_{j \in \IN} : \sPara \in
 [\sPara_\star,\sPara^\star]\}\subset\altcollect$ or $	\{(\exp(-j^{\sPara}))_{j \in \IN} :
\sPara \in[\sPara_\star,\sPara^\star]\}\subset\altcollect$
 for
$\sPara_\star<\sPara^\star$ and  $\sPara_\star,\sPara^\star >0$.

\begin{theorem}[\textbf{Adaptive factor for ordinary smoothness and mild ill-posedness}]
	\label{adapt:fact:o:m}
        Let $\altcollect$ be non-trivial with respect
        to polynomial decay for $1/2<\sPara_\star<\sPara^\star$. Let $\lv \epsden[j] \rv \sim
        j ^{-\pPara}$ for $p > 1/2$. For $\alpha \in (0,1)$ there exists an
        $n_\circ \in \IN$ and $\uA_\alpha \in (0, \infty)$ such that for all $n \geq n_\circ$ and $A \in [0,\uA_\alpha]$
	\begin{align*}
		  \inf_{\test} \sup_{\wC \in \altcollect}	\trisk[ \test \mid
		\rwC, A {\tRiN[\wC]{\delta n}}] \geq 1 - \alpha
	\end{align*}
	with $\delta = (\log \log n \vee 1 )^{-1/2}$, i.e. $\delta^{-1}$ is a lower bound for the minimal adaptive factor over $\altcollect$.
\end{theorem}
\begin{proof}[Proof of	\cref{adapt:fact:o:m}]
	We intend to apply \cref{prop:adaptive:lb}. To do so, we
        construct a collection of regularity parameters $\altcollect_{N}:=\{ \mwC{m}\in \altcollect, m \in\nset{N} \}\subset\altcollect$, that satisfies \ref{(C1)}-\ref{(C4)}. \\
        \textbf{Definition of the collection.} 
        By \cref{ill} the minimax radius of testing is of order $\tSRi\sim n^{-e(s)}$ with
        exponent $e(s) := \frac{4 \sPara}{4\sPara + 4\pPara + 1}$  for $s\in[\sPara_\star,\sPara^\star]$.  Since $\altcollect$ is non-trivial with respect to polynomial decay, it contains the corresponding regularity sequences to the exponents in the interval $[e_\star,e^\star] := [e(s_\star), e(s^\star)]$. We define a grid of size $N$ (specified below) on $\mc A$ by placing a linear grid on the interval of exponents. Indeed, for $d:= \frac{e^\star - e_\star}{N}$ we define $\altcollect_{N} := \lcb  \mwC{m} := (j^{-{s_m}})_{j \in \IN} : e(s_m) = e^\star - (m-1)d, m \in\nset{N} \rcb \subseteq \mc A$. \\
        \textbf{Verification of the conditions \ref{(C1)}-\ref{(C4)}.} We define $N :=\floor{
        	\frac{e^\star - e_\star}{4} \frac{\log(\delta n)}{\lv
        		\log(\delta)\rv}}$. Tedious but straight-forward calculations show that \ref{(C1)} is satisfied for $n$ large enough. Moreover, it is easily seen that $\delta^2 \log N \lra 1$ for $n \to \infty$. Hence, $\log N - \delta/2 \lra \infty$ and, thus, also \ref{(C2)} is satisfied for $n$ large enough and $c_\alpha := 1/2$. Concerning \ref{(C3)} we observe that $\sup_{m \in \nset{N}} \sum_{j \in \IN} (\mwC[j]{m})^2 \leq \sup_{s \in [s_\star, s^\star]} \sum_{j \in \IN} j^{-2s} \leq \frac{1}{2s_\star - 1} =: \ssconst /2$. Again, for $n$ large enough the existence of a constant $\eta$ satisfying \ref{(C4)} uniformly over $n$ follows, because for $\mwC{} \sim (j^{-s})_{j \in \IN}$ with $s \in [s_\star, s^\star]$ the terms $\wC[\tDi]^2$ and $     	(\delta n)^{-1}\nu_{\tDi}^2$ are of the same order.      
        	\end{proof}

\begin{theorem}[\textbf{Adaptive factor for super smoothness and mild ill-posedness}]
	\label{adapt:fact:s:m}
	Let $\altcollect$ be non-trivial with respect
	to exponential decay for $0<\sPara_\star<\sPara^\star$. Let $\lv \epsden[j] \rv \sim
	j ^{-\pPara}$ for $p > 1/2$. For $\alpha \in (0,1)$ there exists an
	$n_\circ \in \IN$ and $\uA_\alpha \in (0, \infty)$ such that for all $n \geq n_\circ$ and $A \in [0,\uA_\alpha]$
	\begin{align*}
		\inf_{\test} \sup_{\wC \in \altcollect}	\trisk[ \test \mid
		\rwC, A {\tRiN[\wC]{\delta n}}] \geq 1 - \alpha
	\end{align*}
	with $\delta = (\log \log \log n \vee 1 )^{-1/2}$, i.e. $\delta^{-1}$ is a lower bound for the minimal adaptive factor over $\altcollect$.
\end{theorem}
	
\begin{proof}[Proof of	\cref{adapt:fact:s:m}]
	We intend to apply \cref{prop:adaptive:lb}. To do so, we
	construct a collection of regularity parameters $\altcollect_{N}:=\{ \mwC{m}\in \altcollect, m \in\nset{N} \}\subset\altcollect$, that satisfies \ref{(C1)}-\ref{(C4)}. \\
	\textbf{Definition of the collection.} 
	By \cref{ill} the minimax radius of testing is of order $\tSRi\sim n^{-1} (\log n)^{e(s)}$ with
	exponent $e(s) := \frac{2 \pPara + 1/2}{\sPara}$  for $s\in[\sPara_\star,\sPara^\star]$.  Since $\altcollect$ is non-trivial with respect to exponential decay, it contains the corresponding regularity sequences to the exponents in the interval $[e_\star,e^\star] := [e(s^\star), e(s_\star)]$. We define a grid of size $N$ (specified below) on $\mc A$ by placing a linear grid on the interval of exponents. Indeed, for $d:= \frac{e^\star - e_\star}{N}$ we define $\altcollect_{N} := \lcb  \mwC{m} := (e^{-j^{s_m}})_{j \in \IN} : e(s_m) = e_\star + (m-1)d, m \in\nset{N} \rcb \subseteq \mc A$. \\
	\textbf{Verification of the conditions \ref{(C1)}-\ref{(C4)}.} We define $N :=\floor{
		\frac{e^\star - e_\star}{4} \frac{\log\log(\delta n)}{\lv
			\log(\delta)\rv}}$. Tedious but straight-forward calculations show that \ref{(C1)} is satisfied for $n$ large enough. Moreover, it is easily seen that $\delta^2 \log N \lra 1$ for $n \to \infty$. Hence, $\log N - \delta/2 \lra \infty$ and, thus, also \ref{(C2)} is satisfied for $n$ large enough and $c_\alpha := 1/2$. Concerning \ref{(C3)} we observe that $\sup_{m \in \nset{N}} \sum_{j \in \IN} (\mwC[j]{m})^2 \leq \sup_{s \in [s_\star, s^\star]} \sum_{j \in \IN} e^{-2s} \leq (1/2)^{1/s_\star} \Gamma(1/s_\star +1) =: \ssconst /2$, where $\Gamma$ denotes the Gamma-function. Again, for $n$ large enough the existence of a constant $\eta$ satisfying \ref{(C4)} uniformly over $n$ follows, because for $\mwC{} \sim (e^{-j^s})_{j \in \IN}$ with $s \in [s_\star, s^\star]$ the terms $\wC[\tDi]^2$ and $     	(\delta n)^{-1}\nu_{\tDi}^2$ are of the same order.      
\end{proof}

	Comparing \cref{adapt:fact:o:m} and \cref{adapt:fact:s:m} with \cref{ill:adapt} (for the indirect test) and \cref{d:ill:adapt} (for the direct test) shows that the adaptive factors are minimal. Indeed, in the ordinary smooth -- mildly ill-posed model both the direct and the indirect $\max$-test face a deterioration by a $\log \log n$-factor, which \cref{adapt:fact:o:m} shows to be unavoidable. In the more restrictive setting of super smoothness and mild ill-posedness both tests feature a $\log \log \log n$-factor, which is unavoidable due to \cref{adapt:fact:s:m}. In the ordinary smooth -- severely ill-posed model there is no loss for adaptation visible in the testing radius.


%
\appendix
\section{Appendix: auxiliary results used in \cref{sec:2,sec:4}}
\label{proofs1}
The next two assertions,  a concentration inequality for canonical
U-statistics and a Bernstein inequality, provide  our
key arguments in order to control the deviation of the test
statistics. The first assertion is a reformulation of Theorem 3.4.8 in
 \cite{GineNickl2015}.
\begin{proposition}
  \label{simplified:con}
  Let $\{Y_j\}_{j=1}^n$ be independent and identically distributed
  $[0,1)$-valued random variables, and let $h: [0,1)^2 \to\IR$ be a
  bounded symmetric kernel, i.e. $h(y,\tilde{y}) = h(\tilde y, y)$ for
  all $y, \tilde y\in[0,1)$, fulfilling in addition
  \begin{align}
    \label{canonical}
    \IE\lb h(Y_1, y_2) \rb = 0 \qquad \forall\,  y_2 \in [0,1).
  \end{align} 
  Then there are finite constants $\uSA,\uSB,\uSC$ and $\uSD$ such that
    \begin{multline}\label{U:const}
    \sup_{y_1,y_2\in[0,1)}|h(y_1,y_2)|\leq \uSA,\quad
    \sup_{y_2\in[0,1)}\IE h^2(Y_1,y_2)\leq \uSB^2,\quad
    \IE h^2(Y_1,Y_2)\leq \uSC^2 \quad\text{and} \\
    \sup \lcb \IE \big( h(Y_1, Y_2) \He(Y_1) \HeB(Y_2)\big) , \IE
    \He^2(Y_1) \leq 1, \IE \HeB^2(Y_2)  \leq 1 \rcb\leq \uSD
  \end{multline}
  and for all $n \geq 2$ the real-valued canonical U-statistic
  \begin{align*}
    \uS = \frac{1}{n(n-1)} \sum_{l,m\in\nset{n}\atop l\ne m} h(Y_l, Y_m) 
  \end{align*}
  satisfies  for all  $x \geq 0$ 
  \begin{align*}
    \IP\lb \uS  \geq  8 \uSC n^{-1} x^{1/2} + 13 \uSD n^{-1}  x
    + 261 \uSB n^{-3/2} x^{3/2} + 343\uSA n^{-2} x^2 \rb \leq \exp(1-x).
  \end{align*}
\end{proposition}%
 The following version of Bernstein's
inequality can directly be deduced from Theorem 3.1.7 in \cite{GineNickl2015}. 
\begin{proposition}
	\label{bernstein}
	Let $\{Z_j\}_{j=1}^n$ be independent with $|Z_j|\leq \lSb$ almost
        surely and $\IE(|Z_j|^2) \leq \lSv$ for all $j\in\nset{n}$, then for all $x > 0$ and $n \geq 1$, we have
	\begin{align*}
		\IP\lb \frac{1}{n} \sum_{j\in\nset{n}}\lb Z_j - \IE Z_j \rb
          \geq \sqrt{\frac{2 \lSv x}{n}} + \frac{\lSb x}{3n} \rb \leq \exp(-x).
	\end{align*}
      \end{proposition}
\paragraph{Preliminaries.}We eventually calculate  first $\uSA$, $\uSB$ and
$\uSC$ satisfying \eqref{U:const} and exploit that $\uSD:=\uSC$ automatically also fulfils
\eqref{U:const}, which we briefly justify next. Throughout this
section we assume that $\{Y_j\}_{j=1}^n$ are
independent and identically distributed  with Lebesgue-density $\ydf=\xdf\oast\edf\in\Lp[2]$.  We denote by $\gLp$ the set of
(Borel-measurable) functions $\He:[0,1)\to\Rz$ with
$\Vnorm[{\gLp}]{\He}^2:=\int_{[0,1)}\He^2(x)g(x)dx<\infty$. Let  $h: [0,1)^2 \to\Rz$ be a bounded kernel,
i.e. $\Vnorm[{\Lp[\infty]}]{h}:=\sup_{y_1,y_2\in[0,1)}|h(y_1,y_2)|<\infty$,
and define the integral operator $H:\gLp\to\gLp$  with $\He\mapsto H\He$
and $H\He(s):=\int_{[0,1)}h(t,s)\He(t)g(t)dt$ for all
$s\in[0,1)$. Then $H$  is linear and 
 bounded, 
i.e. $\Vnorm[{\gLp\to\gLp}]{H}:=\sup\{\Vnorm[{\gLp}]{H\He}:\Vnorm[{\gLp}]{\He}\leq1\}\leq\IE
h^2(Y_1,Y_2)\leq C^2$ due to the Cauchy-Schwarz inequality. This shows
the claim since its operator norm  satisfies 
\begin{equation}
\Vnorm[{\gLp\to\gLp}]{H}=\sup \lcb \IE(  h(Y_1, Y_2)
\zeta(Y_1) \xi(Y_2)) , \IE \zeta^2(Y_1) \leq 1, \IE \xi^2(Y_2)  \leq 1
\rcb.\label{eq:Hnorm1}
\end{equation}
However, an additional assumption allows us to determine  a slightly different quantity $\uSD$.
For a symmetric kernel the operator $H$ is self-adjoint,
i.e. $\Vskalar[{\gLp}]{H\He,\HeB}=\Vskalar[{\gLp}]{\He,H\HeB}$
for all $\He,\HeB\in\gLp$ using the inner product
$\Vskalar[{\gLp}]{\He,\HeB}:=\int_{[0,1)}\He(s)\HeB(s)g(s)ds$, and we have
\begin{equation}
\Vnorm[{\gLp\to\gLp}]{H}=\sup\{|\Vskalar[{\gLp}]{H\He,\He}|:\Vnorm[{\gLp}]{\He}\leq1\}.\label{eq:Hnorm2}
\end{equation}
The last identity can further be reformulated in terms of a discrete
convolution, which we briefly recall next.
For $p\geq1$ we denote by $\lp[p]:=\lp[p](\Zz)$ the Banach space
of complex sequences over $\Zz$ endowed with
its usual $\lp[p]$-norm given by
$\Vnorm[{\lp[p]}]{\aS}:=\big(\sum_{j\in\Zz}|a_j|^p\big)^{1/p}$
for $\aS:=\Zsuite{a_j}\in\Cz^\Zz$. In the case $p=2$, $\lp[2]$ is a Hilbert
space and the $\lp[2]$-norm is
induced by the inner product
$\Vskalar[{\lp[2]}]{\aS,\bS}:=\sum_{j\in\Zz}a_j\overline{b}_j$
for all $\aS,\bS\in\lp[2]$. For each sequence
$\aS\in\lp[1]$  the discrete convolution operator $\aS\ast:
\lp[2]\to\lp[2]$ with  $\bS\mapsto \aS \ast \bS $ and $ (\aS\ast\bS)_j:=
  \sum_{l \in \IZ} a_{j-l} b_l$ for all $j\in\Zz$, is linear and bounded by
$\Vnorm[{\lp[1]}]{\aS}$,
i.e. $\Vnorm[{\lp[2]\to\lp[2]}]{\aS\ast}:=\sup\{\Vnorm[{\lp[2]}]{\aS\ast
  \bS}:\Vnorm[{\lp[2]}]{\bS}\leq1\}\leq
\Vnorm[{\lp[1]}]{\aS}$. Particularly, it holds%
\begin{equation}
\big|\sum_{j\in\Zz}\overline{b}_j\sum_{l\in\Zz}a_{j-l}b_l\big|=|\Vskalar[{\lp[2]}]{\aS
  \ast \bS,\bS}|\leq
\Vnorm[{\lp[1]}]{\aS}\Vnorm[{\lp[2]}]{\bS}^2\quad\text{for all }\aS\in\lp[1]\text{ and
}\bS\in\lp[2].\label{eq:bound:conv}
\end{equation}
Note that the adjoint  of
$\aS\ast$ is a discrete convolution operator
$\aS^\star\ast$ with $\aS[j]^\star:=\oaS[-j]$ for all $j\in\Zz$.
Hence, if in addition $\aS[j]=\oaS[-j]$ for all $j\in\Zz$, then $\aS\ast$ is self-adjoint.
Recall that the real density $\ydf\in\Lp[2]$ admits  Fourier
coefficients $\fydf=\Zsuite{\fydf[j]}$. The coefficients belong to both
$\lp[2]$ by Parseval's identity,
i.e. $\Vnorm[{\Lp[2]}]{\ydf}=\Vnorm[{\lp[2]}]{\fydf }$,
and also to  $\lp[1]$ due to the convolution theorem. Indeed,
$\Zsuite{\fydf[j]=\fxdf[j]\fedf[j]}$ with $\ydf=\xdf\ccon\edf$ and $\xdf,\edf\in\Lp[2]$ implies
$\Vnorm[{\lp[1]}]{\fydf}\leq\Vnorm[{\lp[2]}]{\fxdf}\Vnorm[{\lp[2]}]{\fedf}<\infty$
due to the
Cauchy-Schwarz inequality. Consequently, the discrete convolution 
$\fydf\ast:\lp[2]\to\lp[2]$ is linear, bounded and self-adjoint.
Moreover, for all $\He\in\Lp[2]$ with $|\He|\in\gLp$ and Fourier
coefficients $\fHe=\Zsuite{\fHe[j]}\in\lp[2]$ we note that
$\Vskalar[{\lp[2]}]{\fydf\ast\fHe,\fHe}=\sum_{j\in\Zz}\ofHe[j]\sum_{l\in\Zz}\fydf[j-l]\fHe[l]=\sum_{j\in\Zz}\ofHe[j]\sum_{l\in\Zz}\Ex(\expb[l](Y)\expb[j](-Y))\fHe[l]=\Ex|\He(Y)|^2=\Vnorm[{\gLp}]{\He}^2\geq0$.
Thereby, if for all $\He=\sum_{j\in\Zz}\fHe[j]
e_j\in\Lp[2]$ we  also have $|\He|\in\gLp$,  then $\fydf\ast$ is
non-negative. As a result, there is a non-negative
operator $(\fydf\ast)^{1/2}$ with
$\Vnorm[{\lp[2]}]{(\fydf\ast)^{1/2}\fHe}^2=\Vskalar[{\lp[2]}]{\fydf\ast\fHe,\fHe}=\Vnorm[{\gLp}]{\He}^2$
for all $\fHe\in\lp[2]$, 
which we use frequently in the proofs below.
\subsection{Auxiliary results used in the proof of \cref{quantiles}}
\begin{lemma}\label{prop:u}
Consider $\lcb \yOb[l] \rcb_{l=1}^n \iid \ydf\in\Lp[2]$
and for  $k\in\Nz$ the  kernel $h: [0,1)^2\to\Rz$ given by 
\begin{align*}
 h(y_1,y_2) = \sum_{|j|\in\nset{k}} \frac{(\expb[j](-y_1) - \fydf[j])(\expb[j](y_2) - \ofydf[j])}{|\fedf[j]|^2}, \qquad\forall\, y_1, y_2 \in [0,1),
\end{align*}
which is real-valued, bounded, symmetric and fulfils \eqref{canonical}. 
Let $\nu_k$ 
and
$m_k$ 
as in \eqref{nu} then 
\begin{equation}
  \label{ABCD}
  \uSA= 4\, \nu_k^4,\quad
  \uSB =
  { 3 \Vnorm[{\lp[2]}]{\fydf}\,\nu_k^3}
  \quad\text{and}\quad
 \uSD=\uSC=2\,\Vnorm[{\lp[2]}]{\fydf}\,\nu_k^2 
\end{equation}
satisfy the condition \eqref{U:const} in \cref{simplified:con}.
If, in addition, $\gLp=\RLp$, then
\begin{equation}\label{D:null}
  \uSD= 4 \,\Vnorm[{\lp[1]}]{\fydf}\, m_k^2 
	\end{equation}
 also satisfies the condition \eqref{U:const} in \cref{simplified:con}.
\end{lemma} 
\begin{proof}[Proof of \cref{prop:u}]   We first calculate  quantities
  $\uSA,\uSB$ and $\uSC$  satisfying \eqref{U:const}, then  by the above
  discussion  $\uSD=\uSC$ also satisfies \eqref{U:const}.  First, consider $\uSA$. From
  \begin{equation}\label{prop:u:bew1}
    \Vnorm[{\Lp[\infty]}]{(\expb[j]-\fydf[j])(\expb[-l]-\ofydf[l])}\leq4\quad\text{
      and } |\fedf[j]|\leq1\quad \text{for all }j,l\in\Zz
  \end{equation}
  we  immediately conclude that $ \Vnorm[{\Lp[\infty]}]{h}\leq 4
  \sum_{|j|\in\nset{k}} |\fedf[j]|^{-2}\leq 4\nu_k^4=\uSA$. Next,
  consider $\uSB$. Since $\Ex(\expb[j](-Y_1)\expb[l](Y_1))=\fydf[j-l]$
  for all $j,l\in\Zz$,  we deduce  for arbitrary  $y_2\in[0,1)$ that
  \begin{multline}\label{prop:u:bew2}
    \Ex\big( |h(Y_1,y_2)|^2\big)  = \Var\big(\sum_{|j|\in\nset{k}}\expb[j](-Y_1)
    \frac{(\expb[j](y_2) - \ofydf[j])}{|\fedf[j]|^2}
    \big)\leq \Ex \lv\sum_{|j|\in\nset{k}}\expb[j](-Y_1)
    \frac{(\expb[j](y_2) - \ofydf[j])}{|\fedf[j]|^2}\rv^2\\
    =   \sum_{|j|\in\nset{k}} \frac{(\expb[j](y_2) -
      \ofydf[j])}{|\fedf[j]|^2}\sum_{|l|\in\nset{k}}\fydf[j-l]
    \frac{(\expb[l](-y_2) - \fydf[l])}{|\fedf[l]|^2}
    = \Vskalar[{\lp[2]}]{\aS   \ast \bS,\bS}
  \end{multline}
  where  $\aS[l]:=\fydf[l]\Ind{\{|l|\in\nset{{ 2k}}\}}$  
  and $\bS[l]:=(\expb[l](-y_2) -
  \fydf[l])|\fedf[l]|^{-2}\Ind{\{|l|\in\nset{k}\}}$ for all
  $l\in\Zz$. Making use of \eqref{eq:bound:conv},
  \eqref{prop:u:bew1}, $\Vnorm[{\lp[2]}]{\fydf}\geq1$
  and $ (2k)^{1/2}\leq \nu_k^2$  it follows 
  \begin{equation*}
    \Vskalar[{\lp[2]}]{\aS
      \ast \bS,\bS} \leq \big( \sum_{|j|\in\nset{{2k}}} |\fydf[j]|\big)\sum_{|j|\in\nset{k}}
 \frac{|\expb[j](y_2) - \ofydf[j]|^2}{|\fedf[j]|^4}\leq (4k)^{1/2}
 \big(\sum_{|j|\in\nset{2k}} |\fydf[j]|^2\big)^{1/2}4\nu_k^4\leq  9\Vnorm[{\lp[2]}]{\fydf}^2\nu_k^6 =\uSB^2
  \end{equation*}
 Combining the last bound and \eqref{prop:u:bew2} we see that 
 $\sup_{y_2\in[0,1)}\Ex |h(Y_1,y_2)|^2\leq \uSB^2$. Next, consider
 $\uSC$. Since
 $\Ex(\expb[j](-Y_1)-\fydf[j])(\expb[l](Y_1)-\ofydf[l])=\fydf[j-l]-\fydf[j]\ofydf[l]$
 for all $j,l\in\Zz$,   applying  the
 Cauchy-Schwarz inequality  we obtain
  \begin{equation}\label{prop:u:bew3}
    \Ex |h(Y_1,Y_2)|^2  
    =
    \sum_{|j|\in\nset{k}}\frac{1}{|\fedf[j]|^2}\sum_{|l|\in\nset{k}}
    \frac{|\fydf[j-l]-\fydf[j]\ofydf[l]|^2}{|\fedf[l]|^2}\leq
    2\Vskalar[{\lp[2]}]{\aS   \ast \bS,\bS} + 2
    \nu_k^4\Vnorm[{\lp[2]}]{\fydf^2}^2
  \end{equation}
  where  $\aS[l]:=|\fydf[l]|^2\Ind{\{|l|\in\nset{{2k}}\}}$ and
  $\bS[l]:=|\fedf[l]|^{-2}\Ind{\{|l|\in\nset{k}\}}$ for all
  $l\in\Zz$. Moreover, from $\Vskalar[{\lp[2]}]{\aS   \ast \bS,\bS}\leq
  \Vnorm[{\lp[1]}]{\aS}\Vnorm[{\lp[2]}]{\bS}^2\leq
  \Vnorm[{\lp[2]}]{\fydf}^2\nu_k^4$ due to \eqref{eq:bound:conv} we
  conclude that \eqref{prop:u:bew3} and $|\fydf[j]|\leq1$, $j\in\Zz$,
  together imply  $\Ex |h(Y_1,Y_2)|^2\leq 4
  \nu_k^4\Vnorm[{\lp[2]}]{\fydf}^2=\uSC^2$. Finally, consider $\uSD$
  and assume in addition $\gLp=\RLp$ which allows us
  to use the identities \eqref{eq:Hnorm1} and \eqref{eq:Hnorm2}
   formulated in terms of an operator $H$.  Let $\He\in
 \gLp$, which implies $\He=\sum_{j\in\Zz}\fHe[j]
\expb[j]\in\Lp[2]$.  Exploiting
 $\Ex (\expb[j](-Y_1) - \fydf[j]) \He(Y_1)=(\fydf\ast \fHe)_j-\fydf[j]\Ex\He(Y_1)$ and
 $|\fydf[j]|\leq1$ for all $j\in\Zz$
 straightforward calculations show
 \begin{multline}\label{prop:u:bew4}
   |\Vskalar[{\gLp}]{H\He,\He}|=
   \sum_{|j|\in\nset{k}} \frac{1}{|\fedf[j]|^2} \lv\Ex\lb (\expb[j](-Y_1) - \fydf[j]) \He(Y_1) \rb  \rv^2
 \leq
  m_k^2 \sum_{|j|\in\nset{k}} \lv (\fydf\ast
  \He_{\mbullet})_j-\fydf[j]\Ex\He(Y_1) \rv^2
\\\leq 2 m_k^2 \big(\Vnorm[{\lp[2]}]{\fydf\ast
  \He_{\mbullet}}^2+
\Vnorm[{\lp[1]}]{\fydf}\Vnorm[{\gLp}]{\He}^2\big).
\end{multline}
Using the properties of the discrete convolution recalled above  it follows 
\begin{equation*}
\Vnorm[{\lp[2]}]{\fydf\ast
  \He_{\mbullet}}^2\leq \Vnorm[{\lp[2]\to\lp[2]}]{(\fydf\ast)^{1/2}}^2 \Vnorm[{\lp[2]}]{(\fydf\ast)^{1/2} \He_{\mbullet}}^2=\Vnorm[{\lp[2]\to\lp[2]}]{\fydf\ast}\Vnorm[{\gLp}]{\He}^2\leq\Vnorm[{\lp[1]}]{\fydf}\Vnorm[{\gLp}]{\He}^2
\end{equation*}
which together with \eqref{prop:u:bew4} implies
$|\Vskalar[{\gLp}]{H\He,\He}|\leq 4m_k^2
\Vnorm[{\lp[1]}]{\fydf}\Vnorm[{\gLp}]{\He}^2$ for all
$\He\in\gLp$. We conclude from \eqref{eq:Hnorm2} that 
 $\Vnorm[{\gLp\to\gLp}]{H}\leq4m_k^2 \Vnorm[{\lp[1]}]{\fydf}=\uSD$,
 and finally that $\uSD$ satisfies \eqref{U:const}, by
 \eqref{eq:Hnorm2}, which  completes the proof.
\end{proof}

\begin{lemma}
  \label{linear}Let $\lcb \yOb[l] \rcb_{l=1}^n \iid \ydf=\xdf\ccon\edf\in\Lp[2]$
  with joint distribution $\FuVg{\xdf}$ and
  $\ydfO=\xdfO\oast\edf\in\Lp[2]$.  For each $k\in\Nz$ consider $\qF_k(\xdf - \xdfO)$  and
  $m_k$  as
  in \eqref{qfk} and \eqref{nu}, respectively. Then
  the linear centred statistic $\lSi$ defined in \eqref{decomposition}
  satisfies for all $x \geq1$ and $n \geq 1$
  \begin{align*}
    \FuVg{\xdf} \lb 2\lSi \leq - c\,x^2\,(1\vee m_k^2n^{-1})m_k^2n^{-1} - \tfrac{1}{2} \qF_k(\xdf-\xdfO)  \rb \leq \exp(-x),
  \end{align*}
  where $c=8\Vnorm[{\lp[1]}]{\fydf}+\Vnorm[{\lp[2]}]{\fedf}^2$.
\end{lemma}
\begin{proof}[Proof of \cref{linear}]Introducing the real function $\psi :=
  \sum_{|j|\in\nset{k}}(\fydf[j]-\fydfO[j])|\fedf[j]|^{-2}\expb$ and independent and
  identically distributed random variables
  $Z_j := 2 \psi(Y_j)$, $j\in\nset{n}$,
  we
  intend to apply \cref{bernstein} to $\lSi=\tfrac{1}{n}\sum_{j\in\nset{n}} (Z_j-\FuEx{\xdf}(Z_j))$. Therefore, we compute the required
  quantities $\lSv$ and $\lSb$. First consider $\lSb$. Using subsequently the
  identity $\fydf[l] - \fydfO[l]=(\fxdf[l] - \fxdfO[l])\fedf[l]$,
  $l\in\Zz$, and the Cauchy-Schwarz inequality we deduce that%
  \begin{equation}\label{linear:bew:e1}
    |Z_1|\leq 2\Vnorm[{\Lp[\infty]}]{\psi}\leq 2m_k^2\,
    \sum_{|l|\in\nset{k}} |\fydf[l] - \fydfO[l]|\leq  2m_k^2\,\qFr_k(
    \xdf - \xdfO)\, \Vnorm[{\lp[2]}]{\fedf}=:\lSb.
  \end{equation}
  Secondly, consider $\lSv$.  Since
  $\FuEx{\xdf}(e_j(-Y_1)e_{l}(Y_1))=\fydf[j-l]$ for all $j,l\in\Zz$,
  we see that
  \begin{equation*}
    \FuEx{\xdf}|Z_1|^2= 4\FuEx{\xdf}|\psi(Y_1)|^2 
    = 4 \sum_{|j|\in\nset{k}} \frac{\ofydf[j] -
      \ofydfO[j]}{|\fedf[j]|^2}
    \sum_{|l|\in\nset{k}} \fydf[j-l] \frac{\fydf[l] -
      \fydfO[l]}{|\fedf[l]|^2}
    =4\Vskalar[{\lp[2]}]{\aS   \ast \bS,\bS},
  \end{equation*}
  where $\aS[l]:=\fydf[l]\Ind{\{|l|\in\nset{2k}\}}$ and
  $\bS[l]:=(\fydf[l] - \fydfO[l])|\fedf[l]|^{-2}\Ind{\{|l|\in\nset{k}\}}$
  for all $l\in\Zz$.  Successively  exploiting further
  \eqref{eq:bound:conv} and the identity
  $\fydf[l] - \fydfO[l]=(\fxdf[l] - \fxdfO[l])\fedf[l]$, $l\in\Zz$,
  we conclude that%
  \begin{multline}\label{linear:bew:e2}
    \FuEx{\xdf}|Z_1|^2\leq 4\Vnorm[{\lp[2]}]{\bS}^2\Vnorm[{\lp[1]}]{\aS}=
    4  \sum_{|j|\in\nset{k}}  |\fedf[j]|^{-4}|\fydf[j] - \fydfO[j]|^2
    \sum_{|j|\in\nset{k}} |\fydf[j]|\\
    \leq 4\, m_k^2 \,\qF_k(\xdf - \xdfO)\Vnorm[{\lp[1]}]{\fydf}=:\lSv  .
  \end{multline}
  The claim of \cref{linear} now follows from \cref{bernstein} with
  $\lSb$ and $\lSv$ as in \eqref{linear:bew:e1} and \eqref{linear:bew:e2},
  respectively. Indeed, making use of
  $2ac \leq \frac{a^2}{\eps} + c^2 \eps$ for any $\eps,a, c > 0$,
   we have%
  \begin{equation*}
    \frac{\lSb x}{3n} \leq \eps_1 \qF_k(\xdf - \xdfO) +
                   \frac{x^2}{9 \eps_1} \Vnorm[{\lp[2]}]{\fedf}^2
                   \frac{m_k^4}{n^2}\quad\text{ and }\quad
    \sqrt{\frac{2 \lSv x}{n}}
                  \leq \eps_2 \qF_k(\xdf-\xdfO)
                   + \frac{2x}{\eps_2}  \frac{m_k^2}{n}\Vnorm[{\lp[1]}]{\fydf}.
  \end{equation*}
Combining both bounds (with $\eps_1 = \eps_2 = \frac{1}{4}$) yields for all $x\geq1$
  \begin{equation*}
   \sqrt{\frac{2 \lSv x}{n}}+\frac{\lSb x}{3n} %
\leq  \tfrac{1}{2}\qF_k(\xdf-\xdfO)+c\,x^2(1\vee
\frac{m_k^2}{n})\frac{m_k^2}{n} \text{ with }c=8\Vnorm[{\lp[1]}]{\fydf}+\Vnorm[{\lp[2]}]{\fedf}^2.
  \end{equation*}
  Hence, the assertion follows from \cref{bernstein} by the usual symmetry argument.
\end{proof}
\subsection{Auxiliary results used in the proof of \cref{d:quantiles}}
\begin{corollary}\label{d:prop:u}
Consider $\lcb \yOb[l] \rcb_{l=1}^n \iid \ydf\in\Lp[2]$
and for  $k\in\Nz$ the  kernel $h: [0,1)^2\to\Rz$ given by 
\begin{align*}
 h(y_1,y_2) = \sum_{|j|\in\nset{k}} (e_j(-y_1) - \fydf[j])(e_j(y_2) - \ofydf[j]), \qquad\forall\, y_1, y_2 \in [0,1).
\end{align*}
which is real-valued, bounded, symmetric and fulfils \eqref{canonical}. 
Then the quantities
\begin{equation*}
  \uSA= 8k,\quad
  \uSB =3\,\Vnorm[{\lp[2]}]{\fydf}\, (2k)^{3/4} 
  ,\quad\text{and}\quad
  \uSD=\uSC=2\,\Vnorm[{\lp[2]}]{\fydf}\,(2k)^{1/2}
\end{equation*}
satisfy the condition \eqref{U:const} in \cref{simplified:con}. If, in addition, $\gLp=\RLp$, then
\begin{equation*}
  \uSD= 4 \,\Vnorm[{\lp[1]}]{\fydf}
	\end{equation*}
also satisfies the condition \eqref{U:const} in \cref{simplified:con}.
\end{corollary} 

\begin{proof}[Proof of \cref{d:prop:u}]
Setting  $|\fedf[j]|^2=1$ for all $|j|\in\nset{k}$ the assertion
immediately follows  from \cref{prop:u}.      
\end{proof}

\begin{lemma}
  \label{d:linear}Let $\lcb \yOb[l] \rcb_{l=1}^n \iid \ydf=\xdf\ccon\edf\in\Lp[2]$
  with joint distribution $\FuVg{\xdf}$ and
  $\ydfO=\xdfO\oast\edf\in\Lp[2]$.  For each $k\in\Nz$ consider
  $\qF_k(\ydf - \ydfO)$ as
  in \eqref{qfgk}.  Then
the linear centred statistic $\lSd$ defined in \eqref{d:decomposition}
   satisfies for all $x \geq1$ and $n \geq 1$
  \begin{align*}
    \FuVg{\xdf} \lb 2\lSd
    \leq - c\,x^2\,(1\vee (2k)^{1/2}n^{-1})(2k)^{1/2}n^{-1} - \tfrac{1}{2} \qF_k(\ydf-\ydfO)  \rb \leq \exp(-x)
  \end{align*}
  where $c=12\Vnorm[{\lp[2]}]{\fydf}+1$.
\end{lemma}
\begin{proof}[Proof of \cref{d:linear}]Introducing the real function $\psi :=
  \sum_{|j|\in\nset{k}}(\fydf[j]-\fydfO[j])\expb$ and independent and
  identically distributed random variables
  $Z_j := 2 \psi(Y_j)$, $j\in\nset{n}$, we
  intend to apply \cref{bernstein}  to $\lSd=\tfrac{1}{n}\sum_{j\in\nset{n}} (Z_j-\FuEx{\xdf}(Z_j))$. Therefore, we compute the required
  quantities $\lSv$ and $\lSb$. First consider $\lSb$. Using the
  Cauchy-Schwarz inequality we see that%
  \begin{equation}\label{d:linear:bew:e1}
    |Z_1|\leq 2\Vnorm[{\Lp[\infty]}]{\psi}\leq 2\,
    \sum_{|l|\in\nset{k}} |\fydf[l] - \fydfO[l]|\leq  2 (2k)^{1/2}\,\qFr_k(\ydf - \ydfO)=:\lSb.
  \end{equation}
  Secondly, consider $\lSv$.  Since 
  $\FuEx{\xdf}(e_j(-Y_1)e_{l}(Y_1))=\fydf[j-l]$ for all $j,l\in\Zz$,
  we deduce that
  \begin{equation*}
    \FuEx{\xdf}|Z_1|^2= 4\FuEx{\xdf} |\psi(Y_1)|^2 
    = 4 \sum_{|j|\in\nset{k}} (\ofydf[j] -
      \ofydfO[j]) \sum_{|l|\in\nset{k}} \fydf[j-l] (\fydf[l] -
      \fydfO[l])
    =4\Vskalar[{\lp[2]}]{\aS   \ast \bS,\bS},
  \end{equation*}
  where $\aS[l]:=\fydf[l]\Ind{\{|l|\in\nset{{2k}}\}}$ and
  $\bS[l]:=(\fydf[l] - \fydfO[l])\Ind{\{|l|\in\nset{k}\}}$
  for all $l\in\Zz$. Hence, \eqref{eq:bound:conv} shows that 
  \begin{equation}\label{d:linear:bew:e2}
    \FuEx{\xdf}|Z_1|^2\leq 4\Vnorm[{\lp[2]}]{\bS}^2\Vnorm[{\lp[1]}]{\aS}=
     4\,\qF_k(\ydf - \ydfO)\sum_{|j|\in\nset{{2}k}} |\fydf[j]|=:\lSv  .
  \end{equation}
  The claim of \cref{d:linear} follows now from \cref{bernstein} with
  $\lSb$ and $\lSv$ as in \eqref{d:linear:bew:e1} and \eqref{d:linear:bew:e2},
  respectively. Indeed, exploiting  $2ac \leq \frac{a^2}{\eps} + c^2
  \eps$ for any $\eps,a, c > 0$ we see that
  \begin{multline*}
    \frac{\lSb x}{3n} \leq \eps_1 \qF_k(\ydf - \ydfO) +
                   \frac{x^2}{9 \eps_1} \frac{2k}{n^2}\quad\text{ and }\\
    \sqrt{\frac{2 \lSv x}{n}}
                  \leq \eps_2 \qF_k(\fydf-\ydfO)
                   + \frac{2x}{\eps_2n}  \sum_{|j|\in\nset{{2k}}} |\fydf[j]|
                  \leq \eps_2 \qF_k(\ydf-\ydfO)
                   + \frac{2x}{\eps_2}  \frac{(4k)^{1/2}}{n}\Vnorm[{\lp[2]}]{\fydf}.
  \end{multline*}
  Combining both bounds (with $\eps_1 = \eps_2 = \frac{1}{4}$)
  we get for all $x\geq1$
  \begin{equation*}
   \sqrt{\frac{2 \lSv x}{n}}+\frac{\lSb x}{3n} \leq
   \tfrac{1}{2}\qF_k(\ydf-\ydfO)+cx^2(1\vee
   (2k)^{1/2}n^{-1})(2k)^{1/2}n^{-1}\text{ with }c=12 \Vnorm[{\lp[2]}]{\fydf}+1.%
 \end{equation*}
   Hence, the assertion follows from \cref{bernstein} by the usual symmetry argument.
\end{proof}


\section{Calculations for the illustrations}
\subsection{Calculations for the radius bounds in \cref{ill:adapt}}
 	Firstly, we determine the order of the term $\tSRiKN{\cK}{\delta n}$ by showing that $\tSRiKN{\cK}{n} \sim \tSRiN[\wC]{n}$ and replacing $n$ with $\delta n$. Indeed, we trivially have $ \tSRiN[\wC]{n} \leq 		\tSRiKN{\cK}{n}$.   By defining  $j_\star: = \lceil \tfrac{2}{4 \pPara+ 4 \sPara + 1} \log_2 n \rceil \lesssim  \log(n^2/2)$ (in the ordinary smooth case) respectively $j_\star := \lceil \tfrac{1}{\sPara} \log_2 \log n \rceil \lesssim  \frac{1}{s_\star} \log \log n $ (in the super smooth case), straightforward calculations then show that 	$\tSRiKN{\cK}{ n} \lesssim  \tSRiKN{2^{j_\star}}{n} \lesssim \tSRiN[\wC]{ n}$. 
 Next, we determine the order of the remainder term  $\tSReaN{\delta^2 n}$ by first calculating 
 $\tSReacKN{\IN}{ n} := \min_{k \in \IN} a_k^2 \vee \frac{m_k^2}{n}$, showing $ \tSReacKN{\IN}{n} \sim		\tSReacKN{\cK_g}{ n}$ and then replacing $n$ with $\delta^2 n$. The variance term $\frac{m_k^2}{n}$ is of order $\frac{k^{2\pPara}}{n}$. In the ordinary smooth case the bias term $\wdclass[k]^2$ is of order $k^{-2\sPara}$. Hence, the minimising $k_\star$ satisfies $k_\star \sim n^{\frac{1}{2\sPara+2\pPara}}$, which yields $\tSReacKN{\IN}{n}  \sim n^{-\frac{\sPara}{\sPara+\pPara}}$. We define $j_\star := \lceil \tfrac{1}{2 \pPara+ 2 \sPara } \log_2 n \rceil \lesssim  \log(n^2/2) $. Straightforward calculations show that $\tSReacKN{\cK_2}{ n} \lesssim  \tSReacKN{2^{j_\star}}{ n} \lesssim \tSReacKN{\IN}{n}$.
 Since, trivially $ \tSReacKN{\IN}{n} \leq 		\tSReacKN{\cK_g}{ n} $, we obtain the assertion. In the super smooth case the bias term $\wdclass[k]^2$ is of order $e^{-2k^\sPara}$. Hence, the minimising $k_\star$ satisfies $k_\star \sim \log(n)^{\frac{1}{\sPara}}$, which yields $\tSReacKN{\IN}{n} \sim \frac{1}{n} ( \log n)^{\frac{2\pPara}{\sPara}}$. We define $j_\star := \lceil \tfrac{1}{\sPara} \log_2 \log n \rceil \lesssim  \frac{1}{s_\star} \log \log n $. Straightforward calculations show that 	$\tSReacKN{\cK_{s_\star}}{ n} \lesssim  \tSReacKN{2^{j_\star}}{ n} \lesssim \tSReacKN{\IN}{n}$.
Since, trivially $\tSReacKN{\IN}{n} \leq 		\tSReacKN{\cK_{s_\star}}{ n}$, we obtain the assertion.

\subsection{Calculations for the radius bounds in \cref{d:ill}}
 \textbf{(ordinary smooth - mildly ill-posed)} Since $ \frac{(2k)^{1/2}}{n} m_k^2 \sim \tfrac{1}{n} k^{2 \pPara + 1/2}$ and $\wdclass[k]^2 \sim k^{-2 \sPara}$, the optimal $	\tDd $ satisfies $	\tDd  \sim n^{\frac{2}{4\pPara + 4\sPara + 1}}$, which yields an upper bound of order $  \tSRd  \sim \lb 	\tDd  \rb^{-2 \sPara} \sim  n^{-\frac{4s}{4\pPara + 4\sPara + 1}}$. 
\textbf{(ordinary smooth - severely ill-posed)} 
 Since $ \frac{(2k)^{1/2}}{n} m_k^2 \sim \tfrac{1}{n} k^{1/2} e^{2 k^\pPara}$ and $\wdclass[k]^2 \sim k^{-2 \sPara}$, we obtain  $\tDd \sim (\log n)^{\frac{1}{\pPara}}$, which yields an upper bound of order $ \tSRd  \sim \lb \tDd \rb^{-2 \sPara} \sim  (\log n)^{-\frac{2\sPara}{\pPara}}$. \\
 \textbf{(super smooth - mildly ill-posed)} Since $ \frac{(2k)^{1/2}}{n} m_k^2 \sim\tfrac{1}{n} k^{2 \pPara + 1/2}$ and $\wdclass[k]^2 \sim e^{- 2 k^\sPara}$, we obtain $\tDd \sim (\log n)^{\frac{1}{\sPara}}$, which yields an upper bound of order $ \tSRd \sim \frac{1}{n} \lb \tDd \rb^{2 \pPara + 1/2} \sim  \frac{1}{n} (\log n)^{\frac{2\pPara+1/2}{\sPara}}$.

\subsection{Calculations for the radius bounds in \cref{d:ill:adapt}}
	Firstly, we determine the order of the terms $\tSRdKN{\cK}{\delta n}$ by showing that $\tSRdKN{\cK_2}{n} \sim (\tRdN{ n})^2$ and replacing $n$ with $\delta n$.  Indeed, we trivially have $ \tSRdKN{\cK}{\delta n} \leq (\tRdN{ n})^2$. Define  $j_\star: = \lceil \tfrac{2}{4 \pPara+ 4 \sPara + 1} \log_2 n \rceil \lesssim  \log(n^2/2)$ (ordinary smooth -- mildly ill-posed case), $j_\star := \lceil \tfrac{1}{s} \log_2 \log n \rceil \lesssim  \frac{1}{s_\star} \log \log n$ (super smooth -- mildly ill-posed case) respectively $j_\star := \lceil \tfrac{1}{\sPara} \log_2 \log n \rceil \lesssim  \frac{1}{s_\star} \log \log n $ (ordinary smooth -- severely ill-posed case). Straightforward calculations then show that 	$\tSRiKN{\cK}{ n} \lesssim  \tSRiKN{2^{j_\star}}{n} \lesssim \tSRiN[\wC]{ n}$.
 Next, we determine the order of the remainder term  $\tSReaN{\delta^2 n}$ by first calculating 
$\tSReacKN{\IN}{\delta^2 n} :=  \min_{k \in \IN} a_k^2 \vee \frac{m_k^2}{n}$ and then showing that minimisation over $\mc K_g$ approximates the minimisation over $\IN$ well enough. The calculations in the \textbf{(ordinary smooth - mildly ill-posed)} and \textbf{(super smooth - mildly ill-posed)} cases have already been done in \cref{ill:adapt}. It remains to consider the third case \textbf{(ordinary smooth - severely ill-posed)}. Since $\frac{m_k^2}{n} \sim \frac{e^{2 k^\pPara}}{n}$ and $\wdclass[k]^2 \sim k^{-2\sPara}$, the minimising $k_\star$ satisfies $k_\star^{-2\sPara} \sim\frac{e^{2 k_\star^\pPara}}{n}$ and thus $k_\star \sim (\log n)^{\frac{1}{p}}$, which yields $\tSReacKN{\IN}{n}  \sim (\log n)^{\frac{2s}{p}}$. Next, we show $\tSReacKN{\mc K_g}{n} \sim \tSReacKN{\IN}{n}$. We define $j_\star := \lceil \tfrac{1}{p} \log_2 \log n \rceil \lesssim  \log(n^2/2) $. Straightforward calculations show that
$	\tSReacKN{\cK_g}{ n} \lesssim  \tSReacKN{2^{j_\star}}{ n} \lesssim \tSReacKN{\IN}{n}$.
Since, trivially $ \tSReacKN{\IN}{n} \leq 		\tSReacKN{\cK_g}{ n} $, we obtain the assertion by replacing $n$ with $\delta^2 n$.

\section{Calculations for the $\chi^2$-divergence}
In the proof of \cref{adapt:chi2}  below we apply the following
assertion due to \cite{SchluttenhoferJohannes2020a} (Lemma A.1 in the appendix).
\begin{lemma}
	\label{induction}
	For $k \in \IN$ and for each sign vector $\tau \in \{\pm\}^{k}$ let $J^\tau=(J_j^{\tau_j})_{j \in \nset{k}}\in\Rz^k$. Then, 
	\begin{align*}
		\frac{1}{2^{k}}\sum_{\tau \in \{\pm\}^{k}} \prod_{j\in\nset{k}}J_j^{\tau_j}
		& = \prod_{j\in\nset{k}} \frac{J_j^{-}+J_j^{+}}{2}.
	\end{align*}
\end{lemma}

\begin{lemma}[$\chi^2$-divergence for mixtures over hypercubes over several classes]
	\label{adapt:chi2} 
	Let $\indexset$ be an arbitrary index set of finite
        cardinality $|\indexset|\subset \IN$. For each
        $\indexone \in \indexset$ assume $\mDi[\indexone] \in \IN$ and
        $\mPa{\indexone}\in\lp[2](\Nz)\subset\Rz^\Nz$. For $\signv\in \{\pm \}^{\mDi[s]}$
        define coefficients $\mPa{\indexone,\signv}\in\lp[2](\Nz)$ and
        functions $\yden^{\indexone,\signv} \in \Lp[2]$ by setting
	\begin{align*}
		\mPa[j]{\indexone,\signv[j]}= \begin{cases}
			\signv[j]\mPa[j]{\indexone}  & j\in \nset{\mDi[\indexone]} \\
			0 & \text{otherwise}
		\end{cases} \qquad \text{ and } \qquad
                            \yden^{\indexone, \signv}
                            = \expb[0] +  \sum_{|j|\in\nset{\mDi[s]}} \mPa[|j|]{\indexone,\signv[|j|]} \expb.
	\end{align*} 
Assuming $\yden^{\indexone,\signv} \in  \mc D$ for each $\indexone \in
\indexset$ and $\signv \in \lcb \pm \rcb^{\dimindexone}$, we consider
the mixture $\FuVg{1}$ with probability density $\frac{1}{\lv \indexset \rv} \sum_{\indexone \in \indexset} \lb \frac{1}{2^{\mDi[s]}} \sum_{\signv \in \{\pm\}^{\mDi[s]}} \prod_{j\in\nset{n}} \yden^{\indexone,\signv}(z_j) \rb$, $z_j \in [0,1), j \in \nset{n}$ and denote $\FuVg{0}: =\FuVg{\mathds{1}_{[0,1)}}$.	Then, the $\chi^2$-divergence satisfies
	\begin{align*}
		\chi^2(\FuVg{1},\FuVg{0}) \leq \frac{1}{\lv \indexset \rv^2} \sum_{\indexone, \indextwo \in \indexset} \exp\big( 2 n^2 \sum_{j\in\nset{\mDi[s]\wedge \mDi[t]}} (\mPa[j]{\indexone}\mPa[j]{\indextwo})^2\big) - 1.
	\end{align*}
\end{lemma}
\begin{proof}[Proof of \cref{adapt:chi2}]Recall that $	\chi^2(\FuVg{1}, \FuVg{0}) = \FuEx{0}
  \big(\frac{\dif \FuVg{1}}{\dif \FuVg{0}} ( Z_j,_{j\in\nset{n}})
  \big)^2 - 1$ where $(Z_j)_{j \in \nset{n}}$ are independent with identical marginal density $\expb[0]=\mathds{1}_{[0,1)}$ under $\FuVg{0}$.
	Let $z_j\in [0,1)$, $j\in\nset{n}$, then 
	\begin{align*}
		\frac{\dif \FuVg{1}}{\dif \FuVg{0}} (z_j,_{j\in\nset{n}})= \frac{1}{\lv \indexset \rv} \sum_{\indexone \in \indexset} \big( \frac{1}{2^{\mDi[\indexone]}} \sum_{\signv \in \{\pm\}^{\mDi[\indexone]}} \prod_{j\in\nset{n}}\yden^{\indexone,\signv}(z_j) \big) .
	\end{align*} 
	Squaring, taking the expectation under $\FuVg{0}$ and exploiting the independence yields
	\begin{align*}
		\FuEx{0} \big(	\frac{\dif \FuVg{1}}{\dif \FuVg{0}} (Z_j,_{j\in\nset{n}}) \big)^2= \frac{1}{\lv \indexset \rv^2} \sum_{\indexone, \indextwo \in \indexset}  \frac{1}{2^{\dimindexone}}   \frac{1}{2^{\dimindextwo}} \sum_{ \signv \in \{\pm\}^{\mDi[\indexone]}} \sum_{ \signvb \in \{\pm\}^{\mDi[\indextwo]}} \prod_{j\in\nset{n}} \FuEx{0} \lb \yden^{\indexone,\signv}(Z_j) \yden^{\indextwo, \signvb}(Z_j) \rb \\
		= \frac{1}{\lv \indexset \rv^2} \sum_{\indexone, \indextwo \in \indexset}  \frac{1}{2^{\dimindexone}}   \frac{1}{2^{\dimindextwo}} \sum_{ \signv \in \{\pm\}^{\mDi[\indexone]}} \sum_{ \signvb \in \{\pm\}^{\mDi[\indextwo]}} \lb  \FuEx{0} \lb \yden^{\indexone,\signv}(Z_1) \yden^{\indextwo, \signvb}(Z_1) \rb  \rb^n.
	\end{align*}
Exploiting the orthonormality of
        $(\expb)_{j \in \IZ}$	we  calculate
	\begin{equation*}
		\FuEx{0}\lb  \yden^{\indexone,\signv}(Z_1) \yden^{\indextwo, \signvb}(Z_1) \rb = \int \yden^{\indexone,\signv}(z) \yden^{\indextwo, \signvb}(z) \dif z 
                = 1 +  2 \sum_{ j\in\nset{\mDi[\indexone] \wedge \mDi[\indextwo]}} \mPa[j]{\indexone,\signv[j]}\mPa[j]{\indextwo, \signvb[j]}.
	\end{equation*}
	Applying the inequality $1+x \leq \exp(x)$ for all $x \in \IR$ we obtain
	\begin{align*}
	\FuEx{0} \lb \yden^{\indexone,\signv}(Z_1) \yden^{\indextwo,\signvb}(Z_1) \rb =  1 + 2\hspace*{-2.ex}  \sum_{ j\in\nset{\dimindexone \wedge \dimindextwo}}\hspace*{-1.ex}   \mPa[j]{\indexone,\signv[j]}\mPa[j]{\indextwo, \signvb[j]}  \leq  \exp \big(  2\hspace*{-2.ex}  \sum_{ j\in\nset{\dimindexone \wedge \dimindextwo}}\hspace*{-1.ex}   \mPa[j]{\indexone,\signv[j]} \mPa[j]{\indextwo, \signvb[j]} \big) = \prod_{ j\in\nset{\dimindexone \wedge \dimindextwo}} \hspace*{-2.ex}  \exp \big( 2  \mPa[j]{\indexone,\signv[j]}\mPa[j]{\indextwo, \signvb[j]} \big).
	\end{align*}
	Hence, 
	\begin{align*}
	&	\FuEx{0} \big(	\frac{\dif \FuVg{1}}{\dif \FuVg{0}} (Z_j,_{j\in\nset{n}}) \big)^2 \leq
		 \frac{1}{\lv \indexset \rv^2} \sum_{\indexone, \indextwo \in \indexset}  \frac{1}{2^{\dimindexone}}   \frac{1}{2^{\dimindextwo}} \sum_{ \signv \in \{\pm\}^{\dimindexone}} \sum_{ \signvb \in \{\pm\}^{\dimindextwo}} \prod_{ j\in\nset{\dimindexone \wedge \dimindextwo}} \hspace*{-2.ex}  \exp \big( 2 n  \mPa[j]{\indexone,\signv[j]} \mPa[j]{\indextwo, \signvb[j]}\big),
	\end{align*}
	where we  apply \cref{induction} to the $\signvb$-summation with $J_j^{\signvb[j]} = \exp \lb  2 n  \mPa[j]{\indexone,\signv[j]} \mPa[j]{\indextwo, \signvb[j]} \rb$ and obtain
	\begin{align*}
		\FuEx{0} \big(	\frac{\dif \FuVg{1}}{\dif \FuVg{0}}
          (Z_j,_{j\in\nset{n}}) \big)^2 	& \leq	 	 \frac{1}{\lv
                                         \indexset \rv^2}
                                         \sum_{\indexone, \indextwo
                                         \in \indexset}
                                         \frac{1}{2^{\dimindexone}}
                                         \sum_{ \signv \in \{\pm
                                         \}^{\dimindexone}} \prod_{ j\in\nset{\dimindexone \wedge \dimindextwo}} \hspace*{-2.ex} \frac{\exp \lb - 2 n   \mPa[j]{\indexone,\signv[j]} \mPa[j]{\indextwo}  \rb +\exp \lb  2 n   \mPa[j]{\indexone,\signv[j]} \mPa[j]{\indextwo} \rb }{2}
	\end{align*}
	Again applying \cref{induction} to the $\signv$-summation with $J_j^{\signv[j]} = \tfrac{\exp \lb - 2 n  \mPa[j]{\indexone,\signv[j]}\mPa[j]{\indextwo} \rb +\exp \lb  2 n  \mPa[j]{\indexone,\signv[j]} \mPa[j]{\indextwo} \rb }{2} $ yields
	\begin{multline*}
			\FuEx{0}  \big(	\frac{\dif\FuVg{1}}{\dif \FuVg{0}} (Z_j,_{j\in\nset{n}}) \big)^2
		 \leq 
           \frac{1}{\lv \indexset \rv^2} \sum_{\indexone, \indextwo \in \indexset}\prod_{ j\in\nset{\dimindexone \wedge \dimindextwo}} \hspace*{-2.ex} \frac{\exp \lb - 2 n \mPa[j]{\indexone}  \mPa[j]{\indextwo} \rb +\exp \lb  2 n \mPa[j]{\indexone} \mPa[j]{\indextwo}  \rb }{2}  \\
		=  \frac{1}{\lv \indexset \rv^2} \sum_{\indexone, \indextwo \in \indexset}\prod_{ j\in\nset{\dimindexone \wedge \dimindextwo}} \hspace*{-2.ex} \cosh\lb 2 n  \mPa[j]{\indexone} \mPa[j]{\indextwo}  \rb
	\end{multline*}
	Since $\cosh(x) \leq \exp(x^2/2)$, $x\in\Rz$, we obtain
	\begin{align*}
	\FuEx{0}  \big(	\frac{\dif\FuVg{1}}{\dif \FuVg{0}} (Z_j,_{j\in\nset{n}}) \big)^2
		 	& \leq  \frac{1}{\lv \indexset \rv^2} \sum_{\indexone, \indextwo \in \indexset}  \prod_{ j\in\nset{\dimindexone \wedge \dimindextwo}} \hspace*{-2.ex} \exp\big(  2 n^2 ( \mPa[j]{\indexone}  \mPa[j]{\indextwo} )^2\big) =  \frac{1}{\lv \indexset \rv^2} \sum_{\indexone, \indextwo \in \indexset}  \exp\big(  2 n^2\hspace*{-2.ex}  \sum_{ j\in\nset{\dimindexone \wedge \dimindextwo}} \hspace*{-1.ex} (\mPa[j]{\indexone}  \mPa[j]{\indextwo})^2\big),
	\end{align*}
	which completes the proof.
\end{proof}
\bibliography{lit.bib}
\end{document}